\documentclass[12pt,a4paper,reqno]{amsart}
\usepackage{color}
\usepackage{graphics,epic}
\usepackage{amsmath,amssymb, amsthm,enumerate,longtable}
\usepackage[numbers,sort&compress]{natbib}

\newtheorem{theorem}{Theorem}[section]
\newtheorem*{theorem*}{Theorem}

\newtheorem{lemma}[theorem]{Lemma}
\newtheorem{proposition}[theorem]{Proposition}
\newtheorem{corollary}[theorem]{Corollary}

\theoremstyle{definition}

\newtheorem*{conjecture*}{Conjecture}

\newtheorem{remark}[theorem]{Remark}
\newtheorem{definition}[theorem]{Definition}

\renewcommand{\hat}[1]{\widehat{#1}}

\newcommand{\im}{{\rm Im}\,}

\newcommand{\End}{{\rm End}\,}

\newcommand{\Sym}{{\rm Sym}\,}
\newcommand{\Aut}{{\rm Aut}\,}

\newcommand{\irr}{{\rm Irr}\,}
\newcommand{\Irr}{{\rm Irr}\,}

\newcommand{\Z}{\mathbb{Z}}

\newcommand{\C}{\mathbb{C}}

\def\C{{\mathbb C}}

\def\R{{\mathbb R}}

\def\Z{{\mathbb Z}}

\def\1{{\bf 1}}

\def \End{{\rm End}}

\def \h{\mathfrak{h}}

\def \w{\omega}
\def \g{\mathfrak{g}}

\makeatletter \@addtoreset{equation}{section}

\begin{document}

\title[Automorphism groups of parafermion VOAs: general case]{Automorphism groups of parafermion vertex operator algebras: general case}
\author[C.H. Lam]{Ching Hung Lam}
\address[C.H. Lam]{Institute of Mathematics, Academia Sinica, Taipei 10617, Taiwan}
\email{chlam@math.sinica.edu.tw}
\thanks{C.H. Lam was  partially supported by   NSTC grant  110-2115-M-001-003-MY3   of Taiwan}
\author[X. Lin]{Xingjun Lin}
\address[X. Lin]{School of Mathematics and Statistics, Wuhan University, Wuhan 430072,
China}
\email{linxingjun88@126.com}
\thanks{X.\ Lin was partially supported by  China NSF grant  12171371}
\author[H. Shimakura]{Hiroki Shimakura}
\address[H. Shimakura]{Department of Applied Mathematics, Faculty of Sciences, Fukuoka University, Fukuoka 814-0180, Japan }%
\email {shimakura@fukuoka-u.ac.jp}%
\thanks{H.\ Shimakura was partially supported by JSPS KAKENHI Grant Numbers JP19KK0065 and JP24K06658.}
\begin{abstract}
We complete the program for determining the full automorphism groups of all parafermion vertex operator algebras associated with simple Lie algebras and positive integral levels. 
We show that the full automorphism group of the parafermion vertex operator algebra is isomorphic to the automorphism group of the associated root system for the remaining cases: (i) the level is at least $3$; (ii) the level is $2$ and the simple Lie algebra is non simply laced.
\end{abstract}
\maketitle
\section{Introduction}

Let $\g$ be a (finite-dimensional) simple Lie algebra and let $\hat{\g}$ be the associated affine Lie algebra.
For a positive integer $k$, let $L_{\hat{\g}}(k,0)$ be the simple affine vertex operator algebra associated with ${\g}$ at level $k$. Let $\h$ be a Cartan subalgebra of $\g$. 
Then the vertex operator subalgebra generated by $\h$ is isomorphic to the Heisenberg vertex operator algebra $M_{{\h}}(k,0)$. The parafermion vertex operator algebra $K(\g,k)$ is the commutant of $M_{{\h}}(k,0)$ in $L_{\hat{\g}}(k,0)$.

For arbitrary $\g$ and $k\in\Z_{>0}$, basic properties of $K(\g,k)$ have been established; the $C_2$-cofiniteness was proved in \cite{ALY}; the rationality was proved in \cite{DR} by using \cite{CM}; the irreducible modules were classified in \cite{DR,ADJR}; the fusion products were determined in \cite{ADJR}.
However irreducible $K(\g,k)$-modules have not been described explicitly as they are realized as some substructures inside irreducible $L_{\hat{\g}}(k,0)$-modules via mirror symmetries (\cite{DW3}).
For example, the conformal weights of irreducible $K(\g,k)$-modules have not been determined yet.

The determination of the automorphism group is one of the fundamental problems in the theory of vertex operator algebras and it is natural to ask what $\Aut(K(\g,k))$ is.
By the construction, the automorphism group $\Aut(\Delta_\g)$ of the root system $\Delta_\g$ of $\g$ acts on $K(\g,k)$ as an automorphism group.
Here $\Aut(\Delta_\g)$ is the subgroup of the orthogonal group $O(\R\Delta_\g)$ preserving $\Delta_\g$.
The main question is whether $\Aut(K(\g,k))$ is obtained from $\Aut(\Delta_\g)$ or not.

For the case $\g=sl_2$ and $k\in\Z_{\ge3}$, it follows immediately from the description of generators of $K(sl_2,k)$ in \cite{DLWY} that $\Aut(K(sl_2,k))\cong\Aut(\Delta_{\g})\cong\Z/2\Z$.
For the case $\g=sl_3$ and $k\in\Z_{\ge3}$, it was proved in \cite{W} that $\Aut(K(sl_3,k))\cong\Aut(\Delta_{\g})$ based on the analysis on the weight $2$ and $3$ subspaces of $K(sl_3,k)$.
When the level $k$ is two and $\g$ is simply laced, the parafermion vertex operator algebra $K(\g,k)$ is isomorphic to the commutant $M_Q$ in $V_{\sqrt2Q}^+$ of certain Virasoro vertex operator algebra (see \cite{JLY}).
Here $Q$ is the root lattice of $\g$.
In addition, $\Aut(M_Q)$ has been determined in \cite{LSY}; $\Aut(K(\g,2))\cong\Aut(\Delta_{\g})/\{\pm1\}$ if $\g$ is not of type $E_8$ and $\Aut(K(\g,2))\cong Sp_8(2)$ if $\g$ is of type $E_8$.
Notice that $\Aut(K(\g,2))$ is larger than $\Aut(\Delta_{\g})/\{\pm1\}$ if $\g$ is of type $E_8$.

In this article, we prove that the full automorphism group of $K(\g,k)$ is isomorphic  to $\Aut(\Delta_\g)$ for all the remaining cases:

\begin{theorem}[See Theorem \ref{Thm:main2}]\label{T:intro1}  Let $\g$ be a simple Lie algebra and let $k$ be a positive integer.
Assume that one of the following holds: (1) $k\ge3$; (2) $k=2$ and $\g$ is non simply laced.
Then $\Aut(K(\g,k))\cong \Aut(\Delta_\g)$.
\end{theorem} 

As a consequence, $\Aut(K(\g,k))$ is larger than the subgroup induced from $\Aut(\Delta_\g)$ if and only if $\g$ is of type $E_8$ and $k=2$; it is related to the fact that $L_{\hat{\g}}(k,0)$ has an extra simple current module if and only if $\g$ is of type $E_8$ and $k=2$ (cf.\ \cite{L3}).

Our main strategy is to use the action of $\Aut(K(\g,k))$ on the set $\irr(K(\g,k))_{sc}$ of isomorphism classes of simple current $K(\g,k)$-modules.
One key observation is that all simple current $K(\g,k)$-modules can be realized inside $L_{\hat\g}(k,0)$ unless $\g$ is of type $E_8$ and $k=2$ (\cite{ADJR}), that is, for such $\g$ and $k$, 
$$\irr(K(\g,k))_{sc}=\{M^{0,\beta+kQ_L}\mid \beta+kQ_L\in Q/kQ_L\}, $$   where $Q$ is the root lattice of $\g$ and $Q_L$ is the sublattice of $Q$ generated by long roots (see Proposition \ref{simplecurrent}).  For the definition of $M^{0,\beta+kQ_L}$, see  \eqref{Eq:defM}. 
The simple current $K(\g,k)$-modules $M^{0,\beta+kQ_L}$ associated with roots can be characterized in terms of their conformal weights as follows:  

\begin{theorem}[See Theorem \ref{Thm:key}]\label{T:intro2}  Let $k\in\Z_{\ge2}$.
Let $\g$ be a simple Lie algebra and let $Q$ be the root lattice of $\g$.
Let $Q_L$ be the sublattice of $Q$ generated by long roots.
Let $M^{0,\beta+kQ_L}$ be the simple current $K(\g,k)$-module associated with $\beta+kQ_L\in Q/kQ_L$.
Let $r$ be the lacing number of $\g$.
Then for $t\in\{1,r\}$, $\rho(M^{0,\beta+kQ_L})=1-\dfrac{1}{tk}$ if and only if $\beta+kQ_L$ contains a root of norm $2/t$.
\end{theorem}

For the proof of Theorem \ref{T:intro2}, we consider the charge $\gamma$ subspace
\[ 
L_{\hat\g}(k,0)(\gamma):=\{v\in L_{\hat\g}(k,0)\mid h(0)v=\langle h, \gamma\rangle v\ \text{for all}\ h\in\mathfrak{h}\}
\] of $L_{\hat\g}(k,0)$ with respect to $\mathfrak{h}$, where $\h$ is a (fixed) Cartan subalgebra and $\gamma\in Q$. 
By the construction of $L_{\hat{\g}}(k,0)$, 
the conformal weights of elements in $L_{\hat\g}(k,0)(\gamma)$ are at least $\ell=\ell_Q(\beta)$, which is  
the minimum number $\ell$ such that the sum of $\ell$ elements of $\Delta_\g$ is $\beta$; we call it the \emph{length} (see Definition \ref{Def:length}). 
Some lower bounds of the length $\ell_Q(\beta)$ can be obtained based on the standard descriptions of $\Delta_\g$ for all root systems.
By the construction of $M^{0,\beta+kQ_L}$, the conformal weight of $M^{0,\beta+kQ_L}$ is at least $\ell_Q(\gamma)-\langle\gamma,\gamma\rangle/2$ for any $\gamma\in \beta+kQ_L$.
Then, choosing a suitable representative $\gamma$ of $\beta+kQ_L$ and evaluating $\ell_Q(\gamma)-\langle\gamma,\gamma\rangle/2$, 
we can prove Theorem \ref{T:intro2}.

Next we explain our proof of Theorem \ref{T:intro1} when $\g$ is simply laced and $k\ge3$.
Let $$\varphi:\Aut(\Delta_\g)\to \Aut(K(\g,k))$$ be the group homomorphism mentioned before.
By the description of $\irr(K(\g,k))_{sc}$, $\varphi$ is injective if $\Aut(\Delta_\g)$ acts faithfully on $Q/kQ_L$ (cf. Lemma \ref{L:inj0}).  
By $k\ge3$, any element in $Q/kQ_L$ contains at most one root, which implies that $\varphi$ is injective.

Since $\Aut(\Delta_\g)$ is finite, it suffices to construct an injective group homomorphism from $\Aut(K(\g,k))$ to $\Aut(\Delta_\g)$.
By Theorem \ref{T:intro2}, the action of $\sigma\in \Aut(K(\g,k))$ on $\irr(K(\g,k))_{sc}$ induces a permutation $\tilde{\sigma}$ on the set $\Delta_\g$ since any element in $Q/kQ_L$ contains at most one root.
Since the action of $\sigma$ on $\irr(K(\g,k))_{sc}$ preserves the fusion product, $\tilde{\sigma}$ preserves the inner products on $\Delta_\g$.
This $\tilde{\sigma}$ can be extended to a linear automorphism of $\R \Delta_\g$ preserving the inner product, which implies $\tilde{\sigma}\in \Aut(\Delta_\g)$ (see Lemma \ref{L:symdelta}).
Hence we obtain a group homomorphism from $\Aut(K(\g,k))$ to $\Aut(\Delta_\g)$.
This is also injective; indeed, if $\sigma\in\Aut(K(\g,k))$ fixes every elements in $\Delta_\g$, then $\sigma$ acts trivially on $\irr(K(\g,k))_{sc}$.
Such an element can be obtained from the restriction of the stabilizer of $\h$ in $\Aut(L_{\hat{\g}}(k,0))$.
This stabilizer acts faithfully on $K(\g,k)$ as $\Aut(\Delta_\g)$.
Hence $\sigma=id$.
Therefore we obtain an injective group homomorphism from $\Aut(K(\g,k))$ to $\Aut(\Delta_\g)$.
When $\g$ is non simply laced, we can use a similar argument for the subset consisting of all short roots in $\Delta_\g$.

This article is organized as follows: In Section 2, we recall some basic facts about vertex operator algebras and parafermion vertex operator algebras. In Section 3, we introduce the notion of \emph{length} and study the length of the elements in the root lattice associated with a root system. In Section 4, we study the conformal weights of simple current $K(\mathfrak{g},k)$-modules and prove one of our key results in Theorem \ref{Thm:key}. In Section 5, we prove our main result in Theorem \ref{Thm:main2} and determine the automorphism groups of parafermion vertex operator algebras.

\begin{center}
{\bf Notations}
\begin{small}
\begin{longtable}{ll} 
$\langle\cdot,\cdot\rangle$& 
the normalized Killing form of $\g$, i.e., $\langle\theta, \theta\rangle=2$ for the highest root $\theta$ \\ & 
or the inner product on the Euclidean space.\\
$|\alpha|^2$& the (squared) norm of a vector $\alpha$, that is, $|\alpha|^2=\langle\alpha,\alpha\rangle$\\
$\Aut(\Delta_\g)$&the subgroup of the orthogonal group $O(\R\Delta_\g)$ preserving $\Delta_\g$.\\
$\Delta_{X_n}$,$\Delta_\g$& the set of all roots in the root system of type $X_n$ or of $\g$.\\
$\g$& a (finite-dimensional) simple Lie algebra over $\C$.\\
$\h$& a Cartan subalgebra of $\g$.\\
$\irr(K(\g,k))_{sc}$& the set of all isomorphism classes of simple current $K(\g,k)$-modules.\\
$K(\g,k)$& the parafermion vertex operator algebra associated with $\g$ and $k\in\Z_{>0}$.\\
$\ell_Q(\beta)$& the length of $\beta$ in a root lattice $Q$ with respect to $\Delta_\g$ (see Definition \ref{Def:length}).\\
$L_{\hat{\g}}(k,0)$& the simple affine vertex operator algebra associated with $\g$ and $k\in\Z_{>0}$.\\
$M_{{\h}}(k,0)$& the Heisenberg vertex operator algebra generated by $\h$.\\
$\rho(M)$& the conformal weight of a module $M$.\\
$Q_L$& the sublattice of a root lattice $Q$ generated by long roots.\\
$V_{\sqrt{k}Q_L}$& the lattice vertex operator algebra associated with the even lattice $\sqrt{k}Q_L$.\\
$\Z_{\ge a}$&the set of all integers at least $a$.\\
\end{longtable}
\end{small}
\end{center}

\section{Preliminaries}

\subsection{Root systems and their automorphism groups}
In this subsection, we discuss some results related to root systems, which will be used later.
For the definition of root systems, see \cite[Chapter III]{H}.

Let $\Delta$ be a root system.
Let $\langle\cdot,\cdot\rangle$ be the positive-definite bilinear form on the Euclidean space $\R\Delta$ spanned by $\Delta$.
Let $Q$ be the root lattice generated by $\Delta$ and let $Q_L$ be the sublattice generated by long roots of $\Delta$.
For $\gamma\in Q$, the value $|\gamma|^2=\langle\gamma,\gamma\rangle$ is called the (squared) \emph{norm} of $\gamma$.

\begin{lemma}\label{L:minnorm} Let $\Delta$ be an irreducible root system.
Assume that the norm of any long root of $\Delta$ is $2$.
Let $k$ be a positive integer.
\begin{enumerate}[{\rm (1)}]
\item If $k\ge3$, then for $\alpha\in\Delta$, $(\alpha+kQ_L)\cap\Delta=\{\alpha\}$.
\item If $k=2$ and the type of $\Delta$ is non simply laced, then for a short root $\alpha\in\Delta$, $(\alpha+kQ_L)\cap\Delta=\{\alpha\}$.
\end{enumerate}
\end{lemma}
\begin{proof} 
Let $\alpha\in\Delta$ and let $\beta\in Q_L$ with $\beta\neq0$.
Then $|\alpha|^2\le|\beta|^2$.
Since the norm of any vector in $\alpha+kQ_L$ belongs to $|\alpha|^2+2k\Z$,  we have 
\[
(\alpha+kQ_L)\cap\Delta\subset\{\gamma\in \alpha+kQ_L\mid |\gamma|^2=|\alpha|^2\}.
\]
It follows from $|\langle\alpha,\beta\rangle|\le |\beta|^2$ and $k\ge2$ that $|\alpha+k\beta|^2-|\alpha|^2=k\langle\beta,2\alpha+k\beta\rangle=0$ if and only if $k=2$ and $\alpha=-\beta$.
If $k\ge3$, then $\alpha+k\beta\notin\Delta$, which shows (1).
If $k=2$, then $\Delta$ is non simply laced and $\alpha$ is a short root by the assumptions.
Hence $\alpha\neq-\beta\in Q_L$, and $\alpha+k\beta\notin\Delta$, which shows (2). 
\end{proof}

Recall that the automorphism group $\Aut(\Delta)$ of $\Delta$ is given by 
\begin{equation}
\Aut(\Delta)=\{g\in O(\R\Delta)\mid g(\Delta)=\Delta\},\label{Eq:defautr}
\end{equation}
where $O(\R\Delta)$ is the orthogonal group of $\R\Delta$ preserving $\langle\cdot,\cdot\rangle$.
The following lemma is immediate from \cite[Section 12.2]{H} and the explicit description of root systems of non simply laced type (cf.\ \cite[Section 12.1]{H}).
\begin{lemma}\label{L:auts} Let $\Delta$ be an irreducible root system of non simply laced type.
Let $\Delta_s$ be the (scaled) root system consisting of all short roots in $\Delta$.
\begin{enumerate}[{\rm (1)}]
\item If the type of $\Delta$ is $B_n,C_n,F_4,G_2$ $(n\ge2)$, then that of $\Delta_s$ is $A_1^n$, $D_n$, $D_4$, $A_2$, respectively, where $D_2=A_1^2$ and $D_3=A_3$.
\item $\Aut(\Delta_s)\cong\Aut(\Delta)$ if and only if the type of $\Delta$ is not $C_4$.
\item If the type of $\Delta$ is $C_4$, then $\Aut(\Delta)$ is an index $3$ subgroup of $\Aut(\Delta_s)$.
\end{enumerate}
\end{lemma}
 
Let $\Sym\Delta$ denote the permutation group on $\Delta$.
The following lemma is probably well-known; however, we include a proof for completeness.

\begin{lemma}\label{L:symdelta} Let $\Delta$ be a (not necessarily irreducible) root system of rank $n$.
Let $g\in\Sym \Delta$ such that $\langle g(\alpha_1),g(\alpha_2)\rangle=\langle\alpha_1,\alpha_2\rangle\ \text{for all}\ \alpha_1,\alpha_2\in\Delta$.
Then there exists an element in $\Aut(\Delta)$ whose restriction to $\Delta$ is $g$. 
In other words, $g$ can be extended to an element in $\Aut(\Delta)$.
\end{lemma}

\begin{proof} 
Let $\alpha_1,\alpha_2,\alpha\in\Delta$ such that $\alpha_1+\alpha_2\in\Delta$.
By the assumption on $g$, we have $\langle g(\alpha_1+\alpha_2)-g(\alpha_1)-g(\alpha_2),g(\gamma)\rangle
= \langle \alpha_1+\alpha_2, \gamma\rangle - \langle \alpha_1, \gamma\rangle
- \langle \alpha_2, \gamma\rangle
=0$ for any $\gamma\in\Delta$.
Similarly, we have $\langle g(-\alpha) +g(\alpha),g(\gamma)\rangle=0$ for any $\gamma\in\Delta$.
Since $\langle\cdot,\cdot\rangle$ is non-degenerate on $\R\Delta$, we have 
\begin{equation}
g(\alpha_1+\alpha_2)=g(\alpha_1)+g(\alpha_2)\quad \text{ and } \quad  g(-\alpha)= - g(\alpha).\label{Eq:innersigma}
\end{equation} 

Let $\Phi=\{\beta_1,\dots,\beta_n\}$ be a set of simple roots of $\Delta$.
For any positive (but not simple) root $\alpha$, there exists $\beta_i\in\Phi$ such that $\alpha-\beta_i\in\Delta$ (see \cite[Lemma A in Section 10.2]{H}).
Let $\gamma=\sum_{i=1}^nc_i\beta_i\in\Delta$.
If $\gamma$ is positive, then by \eqref{Eq:innersigma} and an induction on the height, we obtain 
\begin{equation}
g(\gamma)=\sum_{i=1}^nc_ig(\beta_i).\label{Eq:innersigma2}
\end{equation}
If $\gamma$ is negative, then $-\gamma$ is positive.
By \eqref{Eq:innersigma} and \eqref{Eq:innersigma2}, we have
\begin{equation}
g(\gamma)= -g(-\gamma)=  \sum_{i=1}^n c_ig(\beta_i).\label{Eq:innersigma3}
\end{equation}

Now, we define the linear isomorphism $\tilde{g}$ of $\R\Delta$ by $\tilde{g}(\sum_{i=1}^nc_i\beta_i)=\sum_{i=1}^nc_ig(\beta_i)$, $(c_i\in\R)$.
By \eqref{Eq:innersigma2} and \eqref{Eq:innersigma3} , $\tilde{g}(\gamma)=g(\gamma)$ for any $\gamma\in\Delta$;
 in particular,  $\tilde{g}(\Delta)=\Delta$.
In addition, $\tilde{g}$ preserves the inner product on $\R\Delta$.
Hence $\tilde{g}\in\Aut(\Delta)$ and its restriction to $\Delta$ is $g$.
\end{proof}

\begin{lemma}\label{L:innerZ} Let $\Delta$ be an irreducible root system and let $k\in\Z_{\ge2}$.
Assume that the norm of any long root in $\Delta$ is $2$.
Let $g\in\Sym\Delta$ and let $\alpha_1,\alpha_2\in\Delta$ such that
\begin{align}
\dfrac{\langle\alpha_1,\alpha_2\rangle}{k}\equiv\dfrac{\langle g(\alpha_1),g(\alpha_2)\rangle}{k}\pmod\Z\label{Eq:Z},\\
g(-\alpha_i)=-g(\alpha_i)\quad (i=1,2).\label{Eq:-}
\end{align}
Furthermore, we assume that one of the following holds:
\begin{enumerate}[{\rm (1)}]
\item $k\ge3$;
\item $k=2$, $\Delta$ is of non simply raced type and $\alpha_1,\alpha_2$ are short roots. 
\end{enumerate}
Then $\langle\alpha_1,\alpha_2\rangle=\langle g(\alpha_1),g(\alpha_2)\rangle$.
\end{lemma}

\begin{proof} Let $r$ be the lacing number of $\Delta$.
Then $\langle\alpha_1,\alpha_2\rangle\in\{0,\pm1,\pm2\}$ (resp. $\langle\alpha_1,\alpha_2\rangle\in\{0,\pm1/r,\pm2/r\}$) if $\alpha_1$ or $\alpha_2$ is a long root (resp. both $\alpha_1$ and $\alpha_2$ are short roots).

Suppose, for a contradiction, that the desired equation $\langle\alpha_1,\alpha_2\rangle=\langle g(\alpha_1),g(\alpha_2)\rangle$ does not hold.
By the equation \eqref{Eq:Z}, one of the four cases holds; (i) $k=4$, $\langle\alpha_1,\alpha_2\rangle=\pm2$ and $\langle g(\alpha_1),g(\alpha_2)\rangle=\mp2$; (ii) $k=3$, $\langle\alpha_1,\alpha_2\rangle=\pm2$ and $\langle g(\alpha_1),g(\alpha_2)\rangle=\mp1$; (iii) $k=3$, $\langle\alpha_1,\alpha_2\rangle=\pm1$ and $\langle g(\alpha_1),g(\alpha_2)\rangle=\mp2$; (iv) $k=2$, $r=2$ and $\langle\alpha_1,\alpha_2\rangle=\pm1$, $\langle g(\alpha_1),g(\alpha_2)\rangle=\mp1$.
In the cases (i), (ii) and (iv),the former equation implies $\alpha_1=\pm\alpha_2$, which contradicts the latter equation.
In the case (iii), the latter equation implies $g(\alpha_1)=\mp g(\alpha_2)$.
By using \eqref{Eq:-} if necessary, we obtain $\alpha_1=\mp\alpha_2$, which contradicts the former equation.
Therefore we obtain the desired equation.
\end{proof}

\subsection{Vertex operator algebras, modules and automorphisms}
In this subsection, we recall some notion and terminology about vertex operator algebras from \cite{FLM,FHL,LL}. 

A vertex operator algebra $(V, Y, \1, \w)$ over $\C$ is a $\Z$-graded vector space $V=\bigoplus_{m\in\Z} V_m$ over $\C$ equipped with $$Y(v,z)=\sum_{n\in \Z}v_nz^{-n-1}\in(\End\ V)[[z,z^{-1}]],\qquad v\in V,$$ and the vacuum vector $\1$ and conformal vector $\w$ of $V$ satisfying a number of conditions (cf.\ \cite{FLM,FHL,LL}).
We also use the standard notation $L(n)$ to denote the component operator of $Y(\w,z)=\sum_{n\in \Z}L(n)z^{-n-2}$. A linear automorphism $\sigma$ of $V$ is called an {\em automorphism} of $V$ if $\sigma(\w)=\w$ and $\sigma(u_nv)=\sigma(u)_n\sigma(v)$ for any $u, v\in V$, $n\in \Z$. Note that $\sigma(\1)=\1$. We will denote the group of all automorphisms of $V$ by $\Aut(V)$.

Let $V$ be a vertex operator algebra.
A {\em $V$-module} $(M,Y_M)$ is a $\C$-graded vector space $M=\bigoplus_{m\in\C}M_m$ equipped
with a linear map $$Y_{M}(v,z)=\sum_{n\in\Z}v_nz^{-n-1}\in (\End\ M)[[z,z^{-1}]],\qquad v\in V$$ 
satisfying a number of conditions (cf. \cite{FHL,LL}). 
Note that for $m\in\C$, $M_m=\{v\in M\mid L(0)v=mv\}$ and $\dim M_m<\infty$.
If $M$ is irreducible, then there exists $\rho\in\C$ such that $M=\bigoplus_{n=0}^\infty M_{\rho+n}$ and $M_\rho\neq0$; $\rho$ is called the \emph{conformal weight} of $M$ and is denoted by $\rho(M)$.

Let $\irr(V)$ be the set of all isomorphism classes of irreducible $V$-modules.
We often identify an element in $\irr(V)$ with its representative.
Assume that the fusion product $\boxtimes$ is defined on $V$-modules (cf.\ \cite{HL}).
An irreducible $V$-module $M_1$ is called a \emph{simple current $V$-module} if for any irreducible $V$-module $M_2$, the fusion product $M_1\boxtimes M_2$ is also an irreducible $V$-module.
Let $\irr(V)_{sc}$ be the set of all isomorphism classes of simple current $V$-modules, which is a subset of $\irr(V)$.
The fusion product gives an abelian group structure on $\irr(V)_{sc}$.
 
Let $M$ be a $V$-module.
For $\sigma\in\Aut(V)$, the \emph{$\sigma$-conjugate module} $(M\circ \sigma,Y_{M\circ\sigma})$ is defined as follows: $M\circ \sigma=M$ as a vector 
space and its vertex operator is given by $Y_{M\circ \sigma}( u,z)= Y_M(\sigma u,z)$ for $u\in V$.
If $M$ is irreducible, then so is $M\circ \sigma$.
Hence $\Aut(V)$ acts on $\irr(V)$; $M\mapsto M\circ \sigma$ for $\sigma\in \Aut(V)$.
Note that for $\sigma\in\Aut(V)$ and $M,M^1,M^2\in\irr(V)$,
\begin{equation}
\rho(M)=\rho(M\circ \sigma),\qquad (M^1\circ\sigma)\boxtimes (M^2\circ\sigma)=(M^1\boxtimes M^2)\circ \sigma.\label{Eq:rhoconf}
\end{equation}
In particular, $\Aut(V)$ acts on $\irr(V)_{sc}$ as an automorphism group of an abelian group. 

\subsection{Parafermion vertex operator algebras and simple current modules}  In this subsection, we recall some known facts about parafermion vertex operator algebras. In particular,  we review the classification of simple current modules and their fusion products from \cite{ADJR}.

Let $\g$ be a (finite-dimensional) simple Lie algebra over $\C$ and let $\langle\, ,\, \rangle$ be the normalized Killing form of $\g$, i.e., $|\theta|^2=\langle\theta, \theta\rangle=2$ for the highest root $\theta$ of $\g$. Fix a Cartan subalgebra $\h$ of $\g$ and denote the corresponding root system by $\Delta_{\g}$ and the root lattice by $Q$. 

Recall that the affine Lie algebra $\hat{\g}$ associated with $\g$ is defined as $\hat{\g}=\g\otimes \C[t^{-1}, t]\oplus \C K$ with Lie brackets
\begin{align*}
&[x(m), y(n)]=[x, y](m+n)+\langle x, y\rangle m\delta_{m+n,0}K,\qquad
[K, \hat\g]=0,
\end{align*}
for $x, y\in \g$ and $m,n \in \Z$, where $x(n)$ denotes $x\otimes t^n$.

Let $L_{\hat\g}(k, 0)$ be the simple affine vertex operator algebra associated with $\hat{\g}$ at level $k\in\Z_{>0}$ (\cite{FZ}).
It is well-known that $L_{\hat\g}(k, 0)$ is a simple, rational, $C_2$-cofinite vertex operator algebra of CFT-type. 
The vertex operator subalgebra of $L_{\hat\g}(k, 0)$ generated by $\{h(-1)\1\mid h\in \h\}$ is isomorphic to the Heisenberg vertex operator algebra $M_{\h}(k,0)$ (cf. \cite{LL}). The \emph{parafermion vertex operator algebra} $K(\g,k)$ is defined to be the commutant \cite{FZ} of $M_{\h}(k,0)$ in $L_{\hat\g}(k, 0)$, that is,
\begin{align*}
K(\g,k)=\{v\in L_{\hat\g}(k, 0)\mid u_iv=0, u\in M_{\h}(k,0), i\geq 0\}.
\end{align*}
Clearly, $K(\g,k)$ is of CFT-type.
It was shown in \cite{DR,ALY} (cf. \cite{CM}) that $K(\g,k)$ is rational and $C_2$-cofinite.
Hence the fusion products are defined on $\irr(K(\g,k))$.

Next we recall  the classification of simple current $K(\g,k)$-modules and the fusion products on $\irr(K(\g,k))_{sc}$ from \cite{ADJR}.
It was shown in \cite{DW3} that the commutant subalgebra of $K(\g,k)$ in $L_{\hat\g}(k, 0)$ is isomorphic to a lattice VOA $V_{\sqrt{k}Q_L}$, where $Q_L$ is the sublattice of $Q$ generated by the long roots.  
Moreover,  $L_{\widehat{\mathfrak{g}}}(k,0)$ is a simple current extension of $V_{\sqrt{k}Q_L}\otimes K(\g,k)$ as follows:
\begin{align} L_{\widehat{\mathfrak{g}}}(k,0)\cong \bigoplus_{\beta+kQ_L\in Q/k Q_L}V_{\frac{1}{\sqrt{k}}\beta+\sqrt{k}Q_L}\otimes M^{0,\beta+kQ_L},\label{Eq:Lg1}
\end{align}
where 
\begin{equation}
M^{0,\beta+kQ_L} =\{v\in L_{\hat\g}(k, 0)\mid h(m)v=\langle\beta, h\rangle\delta_{m,0}v \text{ for } h\in \h, m\geq0\}\label{Eq:defM}
\end{equation}
are simple current $K(\g,k)$-modules.
Note that $M^{0,kQ_L}=K(\g,k)$ and that $M^{0,\beta+kQ_L}$ does not depend on the choice of a representative $\beta$ in $\beta+kQ_L$, up to isomorphism.

By \cite[Theorem 5.1]{DR}, any irreducible $K(\g,k)$-module can be obtained from an irreducible $V_{\sqrt{k}Q_L}\otimes K(\g,k)$-submodule of some irreducible $L_{\hat{\g}}(k,0)$-module.
Under certain identification, a complete list of irreducible $K(\g,k)$-modules was given in \cite[Theorem 5.1]{ADJR}. 
In particular, if $(X_n,k)\neq(E_8,2)$, then all non-isomorphic simple current $K(\g, k)$-modules can be obtained from the decomposition \eqref{Eq:Lg1}, where $X_n$ is the type of $\g$. 
\begin{proposition}\label{simplecurrent} Let $\g$ be a simple Lie algebra of type $X_n$ and let $k$ be a positive integer.
If $(X_n,k)\neq (E_8,2)$, then $\{M^{0,\beta+kQ_L}\mid \beta+kQ_L\in Q/kQ_L\}$ is a complete list of all non-isomorphic simple current $K(\g, k)$-modules.
\end{proposition}

The fusion products on $\Irr(K(\g,k))$ were determined in \cite[Theorem 5.2]{ADJR} by using the fusion products of $L_{\hat{\g}}(k,0)$-modules. 
In particular, the fusion products on $\irr(K(\g,k))_{sc}$ are given as follows:

\begin{proposition}\label{Cor:fusion} For $\beta+kQ_L,\gamma+kQ_L\in Q/kQ_L$, we have $$M^{0,\beta+kQ_L}\boxtimes M^{0,\gamma+kQ_L}=M^{0,\beta+\gamma+kQ_L}.$$
\end{proposition}

\begin{remark} For a simple Lie algebra $\g$ of type $E_8$, there exist a simple current $K(\g,2)$-module which is not isomorphic to $M^{0,\beta+2Q}$, $\beta\in Q/2Q$; it is obtained as submodules of an extra simple current $L_{\hat{\g}}(2,0)$-module, $L_{\hat{\g}}(2,\Lambda_7)$ (see \cite{L3}).
\end{remark}

At the end of this subsection, we deal with the automorphism group of $K(\g,k)$  by using $L_{\hat\g}(k,0)$.
Recall that $\Aut(L_{\hat\g}(k, 0))\cong\Aut(\g)$.
Let ${\rm Stab}_{\Aut(L_{\hat\g}(k, 0))}(\h)$ be the stabilizer of $\h$ in $\Aut(L_{\hat\g}(k, 0))$.
Then it also stabilizes $V_{\sqrt{k}Q_L}\otimes K(\g,k)$, and induces a permutation on the set of irreducible $V_{\sqrt{k}Q_L}\otimes K(\g,k)$-modules in the decomposition \eqref{Eq:Lg1}.

From \cite[Section 16.5]{H}, the stabilizer ${\rm Stab}_{\Aut(\g)}(\h)$ is described by the following exact sequence
\begin{equation}
1\to\{\exp({\text ad}\ h)\mid h\in\h\}\to {\rm Stab}_{\Aut(\g)}(\h)\xrightarrow{\psi_0}\Aut(\Delta_\g)\to1,\label{Eq:stab}
\end{equation}
where $\Aut(\Delta_\g)$ is the automorphism group of $\Delta_\g$ (see \eqref{Eq:defautr}) and $\psi_0$ is the restriction of ${\rm Stab}_{\Aut(\g)}(\h)$ to $\h$.
We regard $\h$ as a subspace of $L_{\hat{\g}}(k,0)_1$.
Since $\Aut(L_{\hat{\g}}(k,0))\cong\Aut(\g),$ we obtain the surjective group homomorphism $$\psi:{\rm Stab}_{\Aut(L_{\hat{\g}}(k,0))}(\h)\to\Aut(\Delta_\g)$$
as the restriction of ${\rm Stab}_{\Aut(L_{\hat{\g}}(k,0))}(\h)$ to $\h\subset V_{\sqrt{k}Q_L}$.

By \eqref{Eq:stab}, $\ker\psi=\{\exp(h(0))\mid h\in\h\}\subset \Aut(L_{\hat{\g}}(k,0))$.
By the construction of $K(\g,k)$, $\ker\psi$ acts trivially on $K(\g,k)$.
Hence we can view $\psi$ as the restriction of ${\rm Stab}_{\Aut(L_{\hat{\g}}(k,0))}(\h)$ to $\h\otimes K(\g,k)$.
Note that the action of $\im\psi$ on $K(\g,k)$ is not necessarily faithful.

For $\sigma\in {\rm Stab}_{\Aut(L_{\hat\g}(k, 0))}(\h)$ and $\beta+kQ_L\in Q/kQ_L$, 
\begin{equation*}
\sigma(V_{\frac{1}{\sqrt{k}}\beta+\sqrt{k}Q_L}\otimes M^{0,\beta+kQ_L})=V_{\frac{1}{\sqrt{k}}\psi(\sigma)_{|\h}(\beta)+\sqrt{k}Q_L}\otimes M^{0,\psi(\sigma)_{|\h}(\beta)+kQ_L},
\end{equation*}
where we view $\psi(\sigma)_{|\h}\in\Aut(\Delta_\g)$ as an isometry of $Q$.
The equation above implies that the restriction of $\sigma$ to $V_{\sqrt{k}Q_L}\otimes K(\g,k)$ can be extended to an isomorphism of $V_{\sqrt{k}Q_L}\otimes K(\g,k)$-modules from $V_{\frac{1}{\sqrt{k}}\beta+\sqrt{k}Q_L}\otimes M^{0,\beta+kQ_L}$ to $(V_{\frac{1}{\sqrt{k}}\psi(\sigma)_{|\h}(\beta)+\sqrt{k}Q_L}\otimes M^{0,\psi(\sigma)_{|\h}(\beta)+kQ_L})\circ \sigma_{|V_{\sqrt{k}Q_L}\otimes K(\g,k)}$.
Hence, we obtain
\begin{equation}
(V_{\frac{1}{\sqrt{k}}\beta+\sqrt{k}Q_L}\otimes M^{0,\beta+kQ_L})\circ\sigma_{|V_{\sqrt{k}Q_L}\otimes K(\g,k)}^{-1}\cong V_{\frac{1}{\sqrt{k}}\psi(\sigma)_{|\h}(\beta)+\sqrt{k}Q_L}\otimes M^{0,\psi(\sigma)_{|\h}(\beta)+kQ_L}.\label{Eq:actg}
\end{equation}

\begin{lemma}\label{L:inj} The action of $\Aut(K(\g,k))$ on $\irr(K(\g,k))_{sc}$ is faithful.
\end{lemma}
\begin{proof} Let $\sigma\in\Aut(K(\g,k))$ which acts trivially on $\irr(K(\g,k))_{sc}$.
Set $$\tau=(id,\sigma)\in\Aut(V_{\sqrt{k}Q_L})\times\Aut(K(\g,k))\subset\Aut(V_{\sqrt{k}Q_L}\otimes K(\g,k)).$$
Then $\tau$ preserves every irreducible $V_{\sqrt{k}Q_L}\otimes K(\g,k)$-submodule in \eqref{Eq:Lg1}.
By \cite[Theorem 3.3]{S}, $\tau$ lifts to $\hat{\sigma}\in\Aut(L_{\hat\g}(k, 0))$ satisfying $\hat{\sigma}_{|V_{\sqrt{k}Q_L}\otimes K(\g,k)}=\tau$.
It follows from $\h\subset V_{\sqrt{k}Q_L}$ that $\hat{\sigma}=id$ on $\h$.
Hence $\hat{\sigma}\in{\rm Stab}_{\Aut(L_{\hat\g}(k, 0))}(\h)$, and $\hat{\sigma}\in\ker\psi$.
Hence $\hat{\sigma}_{|K(\g,k)}=id$, and $\sigma=id$.
\end{proof}

As we mentioned above, $\im\psi\cong\Aut(\Delta_\g)$.
Considering the restriction of $\im\psi$ to $K(\g,k)$,  we obtain the group homomorphism   
\begin{equation}
\varphi:\Aut(\Delta_\g)\to\Aut(K(\g,k)).\label{Eq:phi}
\end{equation} 

\begin{lemma}\label{L:inj0} If the action of $\Aut(\Delta_\g)$ on $Q/kQ_L$ is faithful, then $\varphi$ is injective.
\end{lemma}
\begin{proof} Let $g\in\ker\varphi$.
There exists $\hat{g}\in {\rm Stab}_{\Aut(L_{\hat\g}(k, 0))}(\h)$ such that $\psi(\hat{g})=g$.
It follows from $\psi(\hat{g})_{|K(\g,k)}=\varphi(g)=id$ that $M^{0,\beta+kQ_L}\circ\psi(\hat{g})_{|K(\g,k)}\cong M^{0,\beta+kQ_L}$ for all $\beta+kQ_L\in Q/kQ_L$.
The isomorphism \eqref{Eq:actg} shows that $\psi(\hat{g})_{|\h}^{-1}=g^{-1}$ acts trivially on $Q/kQ_L$, and we obtain $g^{-1}=id=g$ by the assumption.
\end{proof}

\section{Lower bounds of the length for root lattices}
In this section, we define the length of elements in a lattice with respect to a finite generating set and give its lower bounds.
As an application, we give lower bounds of the length for irreducible root lattices.

\subsection{Definition of length}

Let $n\in\Z_{>0}$ and set $\Omega_n=\{1,2,\dots,n\}$.

\begin{definition}\label{Def:length} Let $Y$ be a lattice and let $\Pi$ be a finite  generating set of $Y$.
The \emph{length} $\ell_Y(\beta)$ of $\beta$ in $Y$ with respect to $\Pi$ is defined to be $$\ell_Y(\beta)=\min\left\{u\ \left|\ \beta=\sum_{j=1}^u\alpha_j,\ \alpha_j\in\Pi\right.\right\}.$$
\end{definition}

The following lemma is probably well-known; however, we include a proof for completeness.
\begin{lemma}\label{L:lformula} Let $Y$ be a lattice in $\R^n$ and let $\Pi$ be a finite generating set of $Y$.
Let $\{e_i\mid i\in\Omega_n\}$ be a basis of $\R^n$ and let $\beta=\sum_{i=1}^nx_ie_i\in Y$.
Then, for any $\emptyset\neq S\subset\Omega_n$, 
\begin{equation}
\ell_Y(\beta)\ge\dfrac{1}{M_S}\sum_{i\in S}|x_i|
,\quad \text{where}\ M_S=\max\left\{\left.\sum_{i\in S}|a_i|\ \right|\ \sum_{i=1}^na_ie_i\in \Pi\right\}.\label{Eq:ell}
\end{equation}
If the equality holds in \eqref{Eq:ell}, then for $\beta=\sum_{j=1}^{\ell_Y(\beta)}\alpha_j$, $\alpha_j=\sum_{i=1}^ny_{ji}e_i\in\Pi$, we have $\sum_{i\in S}|y_{ji}|=M_S$ for $1\le j\le \ell_Y(\beta)$ and $y_{j'i'}y_{j''i'}\ge0$ for $1\le j',j''\le \ell_Y(\beta)$ and $i'\in S$.
\end{lemma}
\begin{proof} Let $\beta_1,\dots,\beta_u\in\Pi$ such that $\beta=\sum_{j=1}^u\beta_j$.
Let $\beta_j=\sum_{i=1}^nz_{ji}e_i$ for $1\le j\le u$.
Then $\beta=\sum_{j=1}^u\sum_{i=1}^nz_{ji}e_i$ and $x_i=\sum_{j=1}^uz_{ji}$.
Hence 
\begin{equation}
\sum_{i\in S}|x_i|\le\sum_{j=1}^u\sum_{i\in S}|z_{ji}|\le uM_S,\label{Eq:length}
\end{equation}
which shows $u\ge ({1}/{M_S})\sum_{i\in S}|x_i|$ for all $u$.
Therefore $\ell_Y(\beta)\ge \dfrac{1}{M_S}\sum_{i\in S}|x_i|$.

If the equality holds in \eqref{Eq:ell}, 
then all the equalities hold in \eqref{Eq:length}; $\sum_{i\in S}|z_{ji}|=M_S$ for $1\le j\le \ell_Y(\beta)$.
In addition, $|z_{j'i'}+z_{j''i'}|=|z_{j'i'}|+|z_{j''i'}|$ for $1\le j',j''\le \ell_Y(\beta)$ and $i'\in S$, which proves the latter assertion.
\end{proof}

\begin{remark} The rank of the lattice $Y$ in Lemma \ref{L:lformula} is not necessary $n$.
\end{remark}

\subsection{Lower bounds of length for root lattices $X_n\neq E_n$}

Let $X_n$ be the type of an irreducible root system with $X_n\neq E_n$, that is, $X_n\in\{A_n,B_n,C_n,D_n,F_4,G_2\}$. If $X_n\neq C_n$, we use 
$\{e_i\mid i\in\Omega_{n}\}$ to denote an orthonormal basis of $\R^{n}$.
When $X_n=C_n$, we modify the norm of $e_i$ so that $\langle e_i,e_j\rangle=(1/2)\delta_{i,j}$. 
The following sets $\Delta_{X_n}$ denote root systems of the corresponding type $X_n$: 
\begin{align}
\Delta_{A_{n}}&=\{e_i-e_j\mid i, j\in\Omega_{n+1},\ i\neq j\}\quad (n\ge1)
\label{Eq:typeA};\\
\Delta_{B_n}&=\{\pm e_i\mid i\in\Omega_n\}\cup\{\pm(e_i\pm e_j)\mid i, j\in\Omega_n,\ i\neq j\}
\label{Eq:typeB}\quad (n\ge2);\\
\Delta_{C_n}&=\{\pm(e_i\pm e_j)\mid i, j\in\Omega_n,\ i\neq j\}\cup\{\pm 2e_i\mid i\in\Omega_n\}
\label{Eq:typeC}\quad (n\ge2);\\
\Delta_{D_n}&=\{\pm(e_i\pm e_j)\mid i, j\in\Omega_n,\ i\neq j\}
\label{Eq:typeD}\quad (n\ge4); \\
\Delta_{F_4}&=\{\pm e_j, \sum_{i\in\Omega_4}x_ie_i\mid j\in\Omega_4,\ x_i\in\{\pm\frac{1}{2}\}\}\cup \{\pm(e_i\pm e_j)\mid i,j\in\Omega_4,\ i\neq j\}
\label{Eq:typeF};\\
\Delta_{G_2}&=\left\{\left.\pm\frac{1}{3}(2e_a-e_b-e_c)\ \right|\ \{a,b,c\}=\Omega_3\right\}\cup \{e_i-e_j\mid i,j\in\Omega_3,\ i\neq j\}.\label{Eq:typeG}
\end{align}
By abuse of notations, we also use $X_n$ to denote the root lattice of type $X_n$ generated by $\Delta_{X_n}$. 
For $\beta\in X_n$, let $\ell_{X_n}(\beta)$ be the length of $\beta$ with respect to $\Delta_{X_n}$ as defined in Definition \ref{Def:length}.
Substituting $S=\Omega_{n+1}$ if $X_n=A_n$ or $G_2$, $S=\Omega_{n}$ if $X_n=B_n,C_n,D_n$ or $F_4$, and $S=\{a\}$ in Lemma \ref{L:lformula}, we obtain the following:

\begin{proposition}\label{P:length} Let $X_n\in\{A_n,B_n,C_n,D_n,F_4,G_2\}$.
Set $m=n+1,3$ if $X_n=A_n,G_2$, respectively, and set $m=n$ otherwise.
Let $\beta=\sum_{i=1}^{m} x_ie_i\in X_n$.
\begin{enumerate}[{\rm (1)}]
\item $\ell_{X_n}(\beta)\ge({1}/{2})\sum_{i=1}^{m}|x_i|$.
\item If $X_n\neq C_n$, then $\ell_{X_n}(\beta)\ge {|x_a|}$ for any $a\in\Omega_m$.
\end{enumerate}
\end{proposition}

\subsection{Lower bounds of length for the root lattices $E_n$}

Let $P(\Omega_8)$ be the power set of $\Omega_8=\{1,2,\dots,8\}$.
Then $P(\Omega_8)$ has a vector space structure over $\mathbb{F}_2$ under the symmetric difference. 
Let $H_8$ be the subspace of $P(\Omega_8)$ spanned by $$\{1,2,3,4\},\{1,2,5,6\},\{1,2,7,8\},\{1,3,5,7\};$$ it is so-called the extended Hamming code of length $8$.
Note that $H_8$ consists of $\emptyset$, $\Omega_8$ and $14$ elements of size $4$ and that for any $X,Y\in H_8$, $|X\cap Y|\in2\Z$.
For $1\le m\le 8$, let $\binom{\Omega_8}{m}$ denote the set of all size $m$ subsets of $\Omega_8$.
Let $H_8(4)$ denote the set of all size $4$ elements in $H_8$, that is, $H_8(4)=H_8\cap\binom{\Omega_8}{4}$.
The following is well-known (cf. \cite{CS}).
\begin{lemma}\label{L:E8}
Let $T_1\in\binom{\Omega_8}{2}$.
Then there exist $T_2,T_3,T_4\in\binom{\Omega_8}{2}$ such that $T_i\cup T_j\in H_8(4)$ for any distinct $1\le i, j\le 4$.
\end{lemma}

Set $H_7=\{S\in H_8\mid 8\notin S\}$.
Then $H_7$ consists of $\emptyset$ and $7$ elements of size $4$.
Let $H_7(4)$ denote the set of all size $4$ elements in $H_7$.
The following lemma is immediate from Lemma \ref{L:E8}.
\begin{lemma}\label{L:E7}
Let $T_1\in\binom{\Omega_7}{2}$.
Then there exist $T_2,T_3\in\binom{\Omega_7}{2}$ such that $T_i\cup T_j\in H_7(4)$ for any distinct $1\le i, j\le 3$.
\end{lemma}

Let $\{e_i\mid i\in\Omega_8\}$ be a basis of $\R^{8}$ with $\langle e_i,e_j\rangle=2\delta_{i,j}$. 
The following sets $\Delta_{E_n}$, $n=6,7,8$, denote the root systems of type $E_n$:
\begin{align}
\Delta_{E_8}&=\{\pm e_i\mid i\in\Omega_8\}\cup\left\{\left. \sum_{i\in S}c_ie_i\ \right|\  c_i\in\{\pm\frac12\},\ S\in H_8(4)\right\};\label{Eq:typeE}\\
\Delta_{E_7}&=\{\pm e_i\mid i\in\Omega_7\}\cup\left\{\left. \sum_{i\in S}c_ie_i\ \right|\  c_i\in\{\pm\frac12\},\ S\in H_7(4)\right\};\label{Eq:typeE7}\\
\Delta_{E_6}&=\{\pm e_i\mid i\in\Omega_4\}\cup\left\{\left.\sum_{i\in S}c_ie_i\ \right|\ c_i\in\{\pm\frac12\},\ S\in H_7(4),\ \sum_{i\in S\cap\{5,6,7\}}c_i=0\right\}.
\label{Eq:typeE6}
\end{align}
By abuse of notations, we also use $E_n$ to denote the root lattice of type $E_n$ generated by $\Delta_{E_n}$.

\begin{proposition}\label{P:lengthE} Let $n\in\{6,7,8\}$.
Set $m=8,7,7$ if $n=8,7,6$, respectively.
Let $\beta=\sum_{i=1}^m x_ie_i\in E_n$ and set $S=\{i\in\Omega_m\mid x_i\neq0\}$. 
Then 
\begin{equation}
\ell_{E_n}(\beta)\ge\dfrac{1}{M_S}\sum_{i\in S}|x_i|,\quad \text{where}\ M_S=\max\left\{1,\left.\dfrac{|S\cap T|}{2}\ \right| T\in H_m(4)\right\}.\label{Eq:ellEn}
\end{equation}
In particular, the following hold:
\begin{enumerate}[{\rm (1)}]
\item $\ell_{E_n}(\beta)\ge\dfrac{1}{2}\sum_{i\in S}|x_i|$.
\item If $|S|=3$ or $S\in\binom{\Omega_m}{4}\setminus H_m(4)$, then $\ell_{E_n}(\beta)\ge\dfrac{2}{3}\sum_{i\in S}|x_i|$.
\item If $|S|\in \{1,2\}$, then $\ell_{E_n}(\beta)\ge\sum_{i\in S}|x_i|$.
\item If $n\in\{6,7\}$ and $S=\{5,6,7\}$, then $\ell_{E_n}(\beta)\ge\sum_{i\in S}|x_i|$.
\item If $|S|\in\{4,5\}$ and $\ell_{E_n}(\beta)=\dfrac{1}{2}\sum_{i\in S}|x_i|$, then $S\in H_m(4)$ and $|x_j|=|x_{j'}|$  for $j,j'\in S$.
\item If $|S|=3$ and $\ell_{E_n}(\beta)=\dfrac{2}{3}\sum_{i\in S}|x_i|$, then there exists $T\in H_m(4)$ such that $S\subset T$ and $|x_j|=|x_{j'}|$ for $j,j'\in S$.
\end{enumerate}

\end{proposition}
\begin{proof} 
By the definition of $\Delta_{E_n}$, the value $M_S$ in Lemma \ref{L:lformula} is given as in \eqref{Eq:ellEn}.
Hence \eqref{Eq:ellEn} holds and the assertions (1), (2) and (3) follow. 
For (4), we assume that $n\in\{6,7\}$ and $S=\{5,6,7\}$.
It follows from $S\subset\{5,6,7,8\}\in H_8(4)$ that $|X\cap S|\in\{0,2\}$ for any $X\in H_7(4)$.
Hence $M_S=1$ and (4) holds.

Now let us prove (5).
Assume that $|S|\in\{4,5\}$ and $\ell_{E_n}(\beta)=({1}/{2})\sum_{i\in S}|x_i|$.
Set $\beta=\sum_{i=1}^{\ell_{E_n}(\beta)}\alpha_i$, $\alpha_i\in\Delta_{E_n}$.
By Lemma \ref{L:lformula} and \eqref{Eq:typeE}, $\alpha_i=\sum_{j\in T_i}c_je_j$ for some $c_j\in\{\pm1/2\}$ and $T_i\in H_m(4)$ and $T_i\subset S$.
Recall that for any $X,Y\in H_m(4)$, $|X\cap Y|\in2\Z$.
Since $|S|\in\{4,5\}$, we have $S=T=T_i$ for all $i$, and $S\in H_m(4)$.
By Lemma \ref{L:lformula}, $\alpha_i=\alpha_j$ for all $i,j$, and we have (5).
By a similar argument, we also have (6) by using the fact that if $|S|=3$, then there exists at most one $T\in H_m(4)$ such that $S\subset T$ (cf. Lemma \ref{L:E8}).
\end{proof}

\section{Conformal weights of simple current modules}

Let $\g$ be a simple Lie algebra and let $k\in\Z_{>0}$.
In this section, we study the conformal weights of simple current $K(\mathfrak{g},k)$-modules.

Let $\h$ be a Cartan subalgebra.
Let $\langle\cdot ,\cdot \rangle$ be the normalized Killing form of $\g$, i.e., $\langle\theta, \theta\rangle=2$ for the highest root $\theta$.
Let $\Delta_\g$ be the set of all roots of $\g$ and let $Q$ be the root lattice generated by $\Delta_\g$.
We use $Q_{L}$ to denote the sublattice of $Q$ generated by long roots.
Let $L_{\hat\g}(k,0)$ be the simple affine vertex operator algebra associated with $\g$ at level $k$.
For $\gamma\in Q$, let $$L_{\hat\g}(k,0)(\gamma):=\{v\in L_{\hat\g}(k,0)\mid h(0)v=\langle h, \gamma\rangle v\ \text{for all}\ h\in\mathfrak{h}\}$$
be the charge $\gamma$ subspace of $L_{\hat\g}(k,0)$ with respect to $\mathfrak{h}$.
Recall the decomposition \eqref{Eq:Lg1} from Section 2.
For $\beta+kQ_L\in Q/kQ_L$,
\begin{equation*}
V_{\frac{1}{\sqrt k}\beta+\sqrt k Q_L}\otimes M^{0, \beta+kQ_L}=\bigoplus_{\gamma\in\beta+kQ_L}L_{\hat\g}(k,0)(\gamma).
\end{equation*}
Comparing charge $\gamma$ subspaces in the both sides, we obtain, 
\begin{equation}
L_{\hat\g}(k,0)(\gamma)=M_{\mathfrak{h}}(k,\frac{1}{\sqrt k}\gamma)\otimes M^{0,\beta+kQ_L}.\label{Eq:confwt1}
\end{equation}
Here $M_{\mathfrak{h}}(k,\frac{1}{\sqrt k}\gamma)$ is an irreducible $M_{\mathfrak{h}}(k,0)$-submodule of $V_{\frac{1}{\sqrt k}\beta+\sqrt k Q_L}$ and $M_{\mathfrak{h}}(k,0)$ is the Heisenberg vertex operator algebra generated by $\h$.
Note that $\rho(M_{\mathfrak{h}}(k,\frac{1}{\sqrt k}\gamma))={\langle\gamma,\gamma\rangle}/{2k}$ and that $h(0)$ acts on $M_{\mathfrak{h}}(k,\frac{1}{\sqrt k}\gamma)$ as scalar $\langle h,\gamma\rangle$.

For $\alpha\in\Delta_\g$, let $E_\alpha\in\g$ be a root vector associated with $\alpha$.
By the construction (cf.\ \cite{FZ}), $L_{\hat\g}(k,0)(\gamma)$ is spanned by vectors of form
$$h_1(-n_1)\dots h_s(-n_s)E_{\alpha_{s+1}}(-n_{s+1})\dots E_{\alpha_t}(-n_t)\1,$$
where $\sum_{i=s+1}^t\alpha_i=\gamma$, $n_i\in\Z_{>0}$, $h_i\in\mathfrak{h}$ and $\alpha_i\in\Delta_\g$.
As the conformal weight of this vector is $\sum_{i=1}^tn_i$, the lowest conformal weight of $L_{\hat\g}(k,0)(\gamma)$ is at least the length $\ell=\ell_Q(\gamma)$ of $\gamma$ defined in Definition \ref{Def:length}.
In fact for $\gamma=\sum_{i=1}^{\ell}\alpha_i$, $\alpha_i\in\Delta_\g$, the vector $E_{\alpha_{1}}(-1)\dots E_{\alpha_{\ell}}(-1)\1\in L_{\hat\g}(k,0)(\gamma)$ has conformal weight $\ell$ if this vector is non-zero.
In particular, for $\alpha\in\Delta_\g$, the lowest conformal weight of $L_{\hat\g}(k,0)(\alpha)$ is one as $E_\alpha(-1)\1\neq0$.
Comparing the conformal weights of both sides in \eqref{Eq:confwt1}, we obtain the following proposition.

\begin{proposition}\label{P:lowest}Let $k\in\Z_{>0}$ and let $\beta+kQ_L\in Q/kQ_L$.

\begin{enumerate}[{\rm (1)}]
\item $\rho(M^{0,\beta+kQ_L})\in-\dfrac{|\beta|^2}{2k}+\Z$.
\item For any $\gamma\in\beta+kQ_L$, $\rho(M^{0,\beta+kQ_L})\ge \ell_Q(\gamma)-\dfrac{|\gamma|^2}{2k}$. 
\item If $\beta+kQ_L$ contains a root $\gamma\in\Delta_\g$, then $\rho(M^{0,\beta+kQ_L})=1-\dfrac{|\gamma|^2}{2k}$.
\end{enumerate}
\end{proposition}

In the rest of this section, we will prove the following theorem, which
will play important roles  in determining the automorphism group of $K(\g,k)$ in Section \ref{S:5}.

 \begin{theorem}\label{Thm:key} Let $k\in\Z_{\ge2}$.
Let $\g$ be a simple Lie algebra and let $Q$ be the root lattice of $\g$.
Let $Q_L$ be the sublattice of $Q$ generated by long roots; note that $Q=Q_L$ if $\g$ is simply laced.
Let $M^{0,\beta+kQ_L}$ be the simple current $K(\g,k)$-module associated with $\beta+kQ_L\in Q/kQ_L$.
Let $r$ be the lacing number of $\g$.
Then for $t\in\{1,r\}$, $\rho(M^{0,\beta+kQ_L})=1-\dfrac{1}{tk}$ if and only if $\beta+kQ_L$ contains a root of norm $2/t$.
\end{theorem}

By the former equation in \eqref{Eq:rhoconf} and this theorem, we obtain the following: 
\begin{corollary}\label{C} Let $t\in\{1,r\}$ and let $\Delta_t$ be the subset of $\Delta_\g$ consisting of roots of norm $2/t$.
Then the action of $\Aut(K(\g,k))$ on $\irr(K(\g,k))_{sc}$ preserves $\{M^{0,\gamma+kQ_L}\mid \gamma\in\Delta_t\}$.
\end{corollary}

We will prove Theorem \ref{Thm:key} in the following subsections based on case-by-case analysis.
Proposition \ref{P:lowest} (3) shows `if part'.
We will show `only if part'. 

\subsection{Type $A_n$ and $D_n$}
Set $X_n=A_n$ ($n\ge1$) or $X_n=D_n$ $(n\ge4)$.
Let $\Delta_{X_n}$ be the root system of type $X_n$ given in \eqref{Eq:typeA} or \eqref{Eq:typeD} and let $Q=X_n$ be the root lattice generated by $\Delta_{X_n}$.
Note that $Q=Q_L$.

\begin{lemma}\label{L:upperAD} Let $k\in\Z_{\ge2}$ and let $\beta+kQ\in Q/kQ$ with $\beta\notin kQ$.
\begin{enumerate}[{\rm (1)}]
\item If $X_n=A_n$, then there exist $S\subset\Omega_{n+1}$ and $\sum_{i\in S}x_ie_i\in \beta+kQ$ such that $|S|\ge2$ and $1\le|x_i|\le k-1$ for all $i\in S$.
\item If $X_n=D_n$, then there exist $S\subset\Omega_n$ and $\sum_{i\in S}x_ie_i\in \beta+kQ$ satisfying one of the following:
\begin{enumerate}[{\rm (i)}]
\item $|S|\ge2$ and $1\le |x_i|\le k-1$ for all $i\in S$;
\item $S=\{a\}$ and $1\le |x_a|\le k$.
\end{enumerate}
\end{enumerate}
\end{lemma}
\begin{proof} Set $m=n+1$ and $m=n$ if $X_n=A_n$ and $X_n=D_n$, respectively.
Let $\gamma=\sum_{i\in S}x_ie_i\in \beta+kQ$, $S\subset\Omega_m$, such that $x_i\neq0$ for $i\in S$.
Set $\varepsilon_i=x_i/|x_i|$ for $ i\in S$.
By the definition of $\Delta_{X_n}$, $|x_i|\in\Z$ for all $i\in S$.

Assume that $X_n=A_n$ and that there exists $ a\in S$ such that $|x_a|\ge k$.
It follows from $\sum_{i\in S}x_i=0$ that $|S|\ge2$ and there exists $b\in S$ such that $\varepsilon_a x_b<0$.
Then $\gamma-\varepsilon_a k(e_a-e_b)\in \beta+kQ$ and the sum of absolute values of all coefficients of this vector is smaller than that of $\gamma$.
Repeating this procedure, we obtain a desired vector.

Assume that $X_n=D_n$.
If $|S|\ge2$ and there exists $ a\in S$ such that $|x_a|\ge k$, then take $b\in S\setminus\{a\}$.
Then $\gamma-k(\varepsilon_ae_a+\varepsilon_be_b)\in \beta+kQ$ and the sum of absolute values of all coefficients of this vector is smaller than that of $\gamma$.
If $S=\{a\}$ and $|x_a|>k$, then $\gamma-2k\varepsilon_a e_a=(x_a-2k\varepsilon_a)e_a$ and $|x_a-2k\varepsilon_a|<|x_a|$.
Repeating this procedure, we find a vector in $\beta+kQ$ satisfying (i) or (ii).
\end{proof}

\begin{proof}[Proof of Theorem \ref{Thm:key} for type $A_n$ and $D_n$] We assume that $\rho(M^{0,\beta+kQ})=1-({1}/{k})$.
Since $\rho(M^{0,kQ})=0$, we have $\beta\notin kQ$.
Set $m=n+1$ and $m=n$ if $X_n=A_n$ and $X_n=D_n$, respectively.

First, we assume that $\beta$ satisfies the conditions in Lemma \ref{L:upperAD} (1) for $X_n=A_n$ or those in Lemma \ref{L:upperAD} (2) (i) for $X_n=D_n$; $\beta=\sum_{i\in S}x_ie_i$, $S\subset\Omega_m$ satisfies $1\le|x_i|\le k-1$ for all $ i\in S$ and $|S|\ge2$.
By Propositions \ref{P:length} and \ref{P:lowest} (2), we obtain
\begin{align*}
&\rho(M^{0,\beta+kQ})-(1-\frac{1}{k})\ge\ell_{X_n}(\beta)-\dfrac{|\beta|^2}{2k}-(1-\frac{1}{k})\\
\ge&\frac{1}{2}\sum_{i\in S}|x_i|-\frac{1}{2k}\sum_{i\in S}|x_i|^2-\frac{k-1}{k}\\
=&\frac{1}{2k}\left(\sum_{i\in S}({k-1}-|x_i|)(|x_i|-1)+(k-1)(|S|-2)\right)\ge0.
\end{align*}
By the assumption $\rho(M^{0,\beta+kQ})=1-({1}/{k})$, all the equalities hold.
Since $k\ge2$, we have $|S|=2$, $|x_i|\in\{1,k-1\}$ for all $i\in S$ and $\ell_{X_n}(\beta)=(1/2)\sum_{i\in S}|x_i|$.
Set $S=\{a,b\}$.
If $X_n=A_n$, then $\beta\in\{\pm(e_a-e_b),\pm(k-1)(e_a-e_b)\}$ since $x_a+x_b=0$.
Hence $\beta+kQ$ contains a root in $\{\pm(e_a-e_b)\}$.
If $X_n=D_n$, then by Lemma \ref{L:lformula}, we have $|x_a|=|x_b|$.
Hence $\beta+kQ$ contains a root in $\{\pm(e_a\pm e_b)\}$.

By Lemma \ref{L:upperAD} (2), the remaining case is that $X_n=D_n$ and $\beta=x_ae_a$, $1\le |x_a|\le k$.
By Proposition \ref{P:length} (2), $\ell_{X_n}(\beta)\ge|x_a|$.
By Proposition \ref{P:lowest} (2), we obtain
\begin{align*}
\rho(M^{0,\beta+kQ})-(1-\frac{1}{k})\ge\dfrac{1}{2k}\left((|x_a|-1)(2k-1-|x_a|)+1\right)>0,
\end{align*}
which contradicts that $\rho(M^{0,\beta+kQ})=1-({1}/{k})$.
\end{proof}

\subsection{Type $B_n$}
Let $\Delta_{B_n}$ $(n\ge2)$ be the root system of type $B_n$ given in \eqref{Eq:typeB} and let $Q=B_n$ be the root lattice generated by $\Delta_{B_n}$.
The sublattice $Q_L$ of $Q$ generated by long roots is isometric to the root lattice of type $D_n$.
By exactly the same argument as in Lemma \ref{L:upperAD} (2), we obtain the following:
\begin{lemma}\label{L:keyB} Set $Q=B_n$ and let $k\in\Z_{\ge2}$.
Let $\beta+kQ_L\in Q/kQ_L$ with $\beta\notin kQ_L$.
Then, there exist $S\subset\Omega_n$ and $\sum_{i\in S}x_ie_i\in \beta+kQ_L$ satisfying one of the following:
\begin{enumerate}[{\rm (i)}]
\item $|S|\ge2$ and $1\le|x_i|\le k-1$ for all $i\in S$;
\item $S=\{a\}$ and $1\le |x_a|\le k$.
\end{enumerate}
\end{lemma}

\begin{proof}[Proof of Theorem \ref{Thm:key} for type $B_n$] We assume that $\rho(M^{0,\beta+kQ_L})=1-({1}/{tk})$ for $t\in\{1,2\}$.
Since $\rho(M^{0,kQ_L})=0$, we have $\beta\notin kQ_L$.
Notice that assertions of Lemmas \ref{L:upperAD} (2) and  \ref{L:keyB} are the same.
Hence if $t=1$, then $\beta+kQ_L$ contains a long root by the same argument for $D_n$.

Assume that $t=2$.
We first assume that $\beta$ satisfies (i) in Lemma \ref{L:keyB}; $\beta=\sum_{i\in S}x_ie_i$, $|S|\ge2$ and $1\le|x_i|\le k-1$ for all $i\in S$.
By Proposition \ref{P:lowest} (1) and $t=2$, 
we have $\langle\beta,\beta\rangle\in2k\Z+1$.
Then $\sum_{i\in S}|x_i|\in2\Z+1$.
By Propositions \ref{P:length} (1) and \ref{P:lowest} (2), we have
\begin{align*}
&\rho(M^{0,\beta+kQ_L})-(1-\frac{1}{2k})
\ge\frac{1}{2}(\sum_{i\in S}|x_i|+1)-\frac{1}{2k}\sum_{i\in S}|x_i|^2-\frac{2k-1}{2k}\\
=&\frac{1}{2k}\left(\sum_{i\in S}(k-1-|x_i|)(|x_i|-1)+(k-1)(|S|-1)\right)> 0,
\end{align*}
which contradicts that $\rho(M^{0,\beta+kQ_L})=1-({1}/{2k})$.

Next, we assume that $\beta$ satisfies (ii) in Lemma \ref{L:keyB}; $\beta=x_ae_a$ and $1\le |x_a|\le k$.
By Propositions \ref{P:length} (2) and \ref{P:lowest} (2),
$$\rho(M^{0,\beta+kQ_L})-(1-\frac{1}{2k})\ge |x_a|-\frac{|x_a|^2}{2k}-\frac{2k-1}{2k}=\dfrac{1}{2k}(|x_a|-1)(2k-1-|x_a|)\ge0.$$
By the assumption $\rho(M^{0,\beta+kQ_L})=1-({1}/{2k})$, all the equalities hold.
Hence $|x_a|=1$, and $\beta$ is a short root.
\end{proof}

\subsection{Type $C_n$}

Let $\Delta_{C_n}$ $(n\ge2)$ be the root system of type $C_n$ given in \eqref{Eq:typeC} and let $Q=C_n$ be the root lattice of type $C_n$ generated by $\Delta_{C_n}$.
Then the sublattice $Q_L$ of $Q$ generated by long roots is isometric to the root lattice of type $A_1^n$.
As $kQ_L=\bigoplus_{i=1}^n2k\Z e_i$, we obtain the following:

\begin{lemma}\label{L:keyC} Set $Q=C_n$ and let $k\in\Z_{\ge2}$.
Let $\beta+kQ_L\in Q/kQ_L$ with $\beta\notin kQ_L$.
Then there exist $\emptyset\neq S\subset\Omega_n$ and $\sum_{i\in S}x_ie_i\in \beta+kQ_L$ such that $1\le|x_i|\le k$ for all $i\in S$.
\end{lemma}
\begin{proof}[Proof of Theorem \ref{Thm:key} for type $C_n$] We assume that $\rho(M^{0,\beta+kQ_L})=1-({1}/{tk})$ for $t\in\{1,2\}$.
Since $\rho(M^{0,kQ_L})=0$, we have $\beta\notin kQ_L$.
By Lemma \ref{L:keyC}, we may assume that $\beta=\sum_{i\in S}x_ie_i$ satisfies $1\le|x_i|\le k$ for all $ i\in S$, and $|S|\ge1$.

First, we assume that $|S|\ge2$.
By Propositions \ref{P:length} (1) and \ref{P:lowest} (2)
\begin{align*}
&\rho(M^{0,\beta+kQ_L})-(1-\frac{1}{tk})
\ge\frac{1}{2}\sum_{i\in S}|x_i|-\frac{1}{4k}\sum_{i\in S}|x_i|^2-\frac{tk-1}{tk}\\
=&\frac{1}{4k}\left(\sum_{i\in S}(2k-1-|x_i|)(|x_i|-1)+(2k-1)(|S|-2)+\frac{4}{t}-2\right)\ge0.
\end{align*}
By the assumption $\rho(M^{0,\beta+kQ_L})=1-({1}/{kt})$, all the equalities hold.
Then $t=2$, $|S|=2$ and $|x_i|=1$ for all $i\in S$.
Hence $\beta$ is a short root.

Next, we assume that $S=\{a\}$.
Then $\beta=x_ae_a$ with $1\le |x_a|\le k$.
By the description of $\Delta_{C_n}$ in \eqref{Eq:typeC}, we have $x_a\in2\Z$ and $\langle\beta,\beta\rangle\in2\Z$.
Then $t=1$ by Proposition \ref{P:lowest} (1).
By Propositions \ref{P:length} (2) and \ref{P:lowest} (2),
$$\rho(M^{0,\beta+kQ_L})-(1-\frac{1}{k})\ge\frac{1}{4k}(2k|x_a|-|x_a|^2-4k+4)=\frac{1}{4k}(|x_a|-2)(2k-2-|x_a|)\ge0.$$
By the assumption $\rho(M^{0,\beta+kQ_L})=1-({1}/{k})$, all the equalities hold.
Hence $|x_a|=2$, and $\beta$ is a long root.
\end{proof}

\subsection{Type $E_n$}
For $n=6,7,8$, let $\Delta_{E_n}$ be the root systems of type $E_n$ given as in \eqref{Eq:typeE}, \eqref{Eq:typeE7}  and \eqref{Eq:typeE6}, respectively.
Let $Q=E_n$ be the root lattice of type $E_n$ generated by $\Delta_{E_n}$.
Note that $Q=Q_L$. 

\begin{lemma}\label{L:upperE} Let $k\in\Z_{\ge2}$ and $n\in\{7,8\}$ and set $Q=E_n$.
Let $\beta+kQ\in Q/kQ$ with $\beta\notin kQ$.
Then there exist $S\subset\Omega_n$ and $\sum_{i\in S}x_ie_i\in \beta+kQ$ such that $0<|x_i|\le k/2$ for all $i\in S$. Moreover,  one of the following  holds:
\begin{enumerate}[{\rm (i)}]
\item $|\{i\in S\mid 0<|x_i|<k/2\}|\ge4$ and $|\{i\in S\mid |x_i|=k/2\}|\le1$;
\item $|S|=4$, $|\{i\in S\mid |x_i|=k/2\}|=1$ and $S\notin H_n(4)$;
\item $|S|\in\{1,2,3\}$.
\end{enumerate}
\end{lemma}
\begin{proof} By $\bigoplus_{i=1}^n k\Z e_i\subset kQ$, there exists $\gamma=\sum_{i\in S}x_ie_i\in \beta+kQ$ such that $0<|x_i|\le k/2$ for all $i\in S$.
Set $\varepsilon_i=x_i/|x_i|$ if $i\in S$ and $\varepsilon_i=1$ if $i\notin S$.
If $|S|\le3$, then (iii) holds.

Set $U=\{i\in S\mid |x_i|=k/2\}$ and assume that $|S|\ge4$ and $|U|\ge2$.
By Lemmas \ref{L:E8} and \ref{L:E7}, there exists $T\in H_n(4)$ such that $|U\cap T|\ge2$ and $|S\cap T|\ge3$.
Then the number of coefficients of $\gamma-(k/2)\sum_{i\in T}\varepsilon_ie_i$ with  absolute value $k/2$ is less than that of $\gamma$.
Repeating this procedure, we may assume that $\gamma$ satisfies $|U|\le1$.
We may also assume $|S|=4$ and $|U|=1$; otherwise, $\gamma$ satisfies (i) or (iii).
If $S\in H_n(4)$, then $\gamma-(k/2)\sum_{i\in S}{\varepsilon_i}e_i$ satisfies (iii).
If $S\notin H_n(4)$, then $\gamma$ satisfies (ii).
\end{proof}

Recall from \eqref{Eq:typeE6} and the definition of $H_7$ that $|T\cap\{5,6,7\}|\in2\Z$ for any $T\in H_7(4)$ and that for $T\in H_7(4)$ and $c_i\in\{\pm1/2\}$ $(i\in T)$, $\sum_{i\in T}c_ie_i\in E_6$ if and only if $\sum_{i\in T\cap\{5,6,7\}}c_i=0$.

\begin{lemma}\label{L:E6k/2} Set $Q=E_6$ and let $k\in\Z_{\ge2}$.
Let $\beta+kQ\in Q/kQ$.
Then there exist $S\subset\Omega_7$ and $\sum_{i\in S} x_ie_i \in \beta+kQ$ such that $0<|x_i| <k/2$ for all $i\in \{5,6,7\}\cap S$ and $0< |x_i|\le k/2$ for all $i\in\{1,2,3,4\}\cap S$.
\end{lemma}
\begin{proof}Since  $k(e_s-e_r)\in kQ$ for any $\{r,s\} \subset \{5,6,7\}$ and $x_5+x_6+x_7=0$, there exist $S\subset\Omega_7$ and $\gamma=\sum_{i\in S} x_ie_i \in \beta+kQ $ such that $0< |x_i| < k$ for all $i\in \{5,6,7\}\cap S$.
Let $\{a,b,c\}=\{5,6,7\}$ so that $|x_a|\geq |x_b|\geq |x_c|$.

Assume that $k/2\le |x_a|<k$.
Since $x_a+x_b+x_c=0$, we have $x_ax_b<0$ and $|x_c|<k/2$.
Let $\varepsilon_i=x_i/|x_i|$ for $i\in S$. 
By Lemma \ref{L:E7}, there exists $\{p,q\}\subset \{1,2,3,4\}$ such that $\{p,q,a,b\} \in H_7(4)$. Set $\sum_{i=1}^7 y_i e_i = \gamma - \varepsilon_a (k/2) (e_p+e_q+e_a-e_b).$
Then $|y_a|$, $|y_b|$ and $|y_c|=|x_c|$ are all less than $k/2$. 
By $\bigoplus_{i=1}^4  k\Z e_i\subset kQ$, we obtain a desired vector.
\end{proof}

\begin{lemma}\label{L:upperE6} Set $Q=E_6$ and let $k\in\Z_{\ge2}$.
Let $\beta+kQ\in Q/kQ$ with $\beta\notin kQ$.
Then there exist $S\subset\Omega_7$ and $\sum_{i\in S}x_ie_i\in \beta+kQ$ satisfying one of the following:
\begin{enumerate}[{\rm (i)}]
\item $|\{i\in S\mid 0<|x_i|<k/2\}|\ge4$, $|\{i\in S\mid |x_i|=k/2\}|\le 1$ and $0<|x_i|\le k/2$ for all $i\in S$;
\item $|S|=4$, $|\{i\in S\mid |x_i|=k/2\}|=1$, $S\notin H_7(4)$ and $0<|x_i|\le k/2$ for all $i\in S$;
\item $|S|\in\{1,2,3\}$ and $0<|x_i|\le k/2$ for all $i\in S$;
\item $S=\{5,6,7\}$, $|\{i\in S\mid k/2\le|x_i|< k\}|=1$ and $|\{i\in S\mid 0<|x_i|<k/2\}|=2$. 
\end{enumerate}
\end{lemma}
\begin{proof}
By Lemma \ref{L:E6k/2}, there exist $\gamma=\sum_{i\in S} x_ie_i \in \beta+kQ$ and $S\subset\Omega_7$ such that $0< |x_i|< k/2$ for all $i\in \{5,6,7\}\cap S$ and $0< |x_i|\le k/2$ for all $i\in\{1,2,3,4\}\cap S$.
For $i\in\Omega_7$, set $\varepsilon_i=x_i/|x_i|$ if $i\in S$ and $\varepsilon_i=1$ if $i\notin S$.

If $|S|\le 3$, then (iii) holds; we assume that $|S|\ge4$.
Set $U=\{i\in S\mid |x_i|=k/2\}$.
By the conditions on $\gamma$, we have $U\subset\{1,2,3,4\}$.
If $|U|=0$, then (i) holds; we also assume that $|U|\ge1$.

If $|U|\ge3$, or $|U|=2$ with $U\subsetneq S\cap\{1,2,3,4\}$, then replacing $\gamma$ by $\gamma-(k/2)\sum_{i=1}^4\varepsilon_ie_i$, we may assume $|U|\le 1$.

If $|U|=2$ and $U=S\cap\{1,2,3,4\}$, then there exists $T\in H_7(4)$ such that $U\subset T$ and $|T\cap\{5,6,7\}|=2$, and there are two cases: (a) $|S|=4$ and (b) $|S|=5$.

(a) $|S\cap\{5,6,7\}|=2$, $ |S\cap T\cap\{5,6,7\}|\in\{1,2\}$ and $\prod_{i\in T\cap\{5,6,7\}}x_i\le0$.
If $ |S\cap T\cap\{5,6,7\}|=2$, then $\prod_{i\in T\cap\{5,6,7\}}x_i<0$, and $\gamma-(k/2)\sum_{i\in T}\varepsilon_ie_i$ satisfies (iii).
If $ S\cap T\cap\{5,6,7\}=\{a\}$, then $T=U\cup\{a,b\}$ for some $b\in \{5,6,7\}\setminus S$.
Then $\gamma-(k/2)(\sum_{i\in U}\varepsilon_ie_i+\varepsilon_ae_a-\varepsilon_ae_b)$ satisfies (iii).

(b) $\{5,6,7\}\subset S$ and $|S\cap T\cap\{5,6,7\}|=2$.
Let $\{a,b\}=T\cap\{5,6,7\}$.
If $x_ax_b<0$, then $\gamma-(k/2)\sum_{i\in T}\varepsilon_ie_i$ satisfies (iii); otherwise, $\gamma-(k/2)(\sum_{i\in U}\varepsilon_ie_i+\varepsilon_ae_a-\varepsilon_ae_b)$ satisfies (iv). 

Assume that $|U|=1$.
If $|S|\ge5$, then (i) holds.
If $|S|=4$ and $S\in H_7(4)$, then $|S\cap\{1,2,3,4\}|\in2\Z$ and $|S\cap\{5,6,7\}|\in2\Z$.
Then $S=\{1,2,3,4\}$ or $\prod_{i\in S\cap\{5,6,7\}}x_i<0$ holds, and $\gamma-(k/2)\sum_{i\in S}\varepsilon_ie_i$ satisfies (iii).
If $|S|=4$ and $S\notin H_7(4)$, then (ii) holds.
\end{proof}

\begin{proof}[Proof of Theorem \ref{Thm:key} for type $E_n$] We assume that $\rho(M^{0,\beta+kQ})=1-({1}/{k})$.
Since $\rho(M^{0,kQ})=0$, we have $\beta\notin kQ$.
If $n=7,8$, then  we may assume that $\beta$ satisfies (i), (ii) or (iii) in Lemmas \ref{L:upperE}; if $n=6$, then we may assume that $\beta$ satisfies (i), (ii), (iii) or (iv) in Lemma \ref{L:upperE6}.
Note that the conditions (i)--(iii) are the same in Lemmas \ref{L:upperE} and \ref{L:upperE6}.
Set $m=8,7,7$ if $n=8,7,6$, respectively.

First, we assume (i); $\beta=\sum_{i\in S}x_ie_i$, $S\subset\Omega_m$ with $0<|x_i|\le k/2$ for $i\in S$.
Note that $|x_i|\in\Z/2$ for all $i\in S$.
Set $V=\{i\in S\mid 0<|x_i|<k/2\}$.
By the condition (i), $|V|\ge4$ and $|S\setminus V|\le 1$.
By Propositions \ref{P:lengthE} (1) and \ref{P:lowest} (2), we obtain
\begin{align*}
&\rho(M^{0,\beta+kQ})-(1-\frac{1}{k})\ge\ell_{E_n}(\beta)-\dfrac{|\beta|^2}{2k}-(1-\frac{1}{k})\\
\ge&\frac{1}{2}\sum_{i\in S}|x_i|-\frac{1}{k}\sum_{i\in S}|x_i|^2-\frac{k-1}{k}\\
=&\frac{1}{2k}\left(2\sum_{i\in V}(\frac{k-1}{2}-|x_i|)(|x_i|-\frac12)+\frac12(k-1)(|V|-4)\right)\ge0.
\end{align*}
By the assumption $\rho(M^{0,\beta+kQ})=1-({1}/{k})$, all the equalities hold.
Hence $|V|=4$ and $|x_i|\in\{1/2,(k-1)/2\}$ for all $i\in V$.
Then $|S|\in\{4,5\}$.
In addition, $\ell_{E_n}(\beta)=(1/2)\sum_{i\in S}|x_i|$.
By Lemma \ref{P:lengthE} (5), $|x_i|$ is constant for all $i\in S$ and $S\in H_m(4)$.
Then $S=V$ and $\beta+kQ$ contains a root in $\{\sum_{i\in S}c_ie_i\mid c_i\in\{\pm1/2\}\}$.

Next, we assume (ii) or $|S|=3$ in (iii); $\beta=\sum_{i\in S}x_ie_i$, $S\subset\Omega_m$ with $0<|x_i|\le k/2$ for $i\in S$, $|S|\in\{3,4\}$ and $S\notin H_m(4)$.
In this case, $|x_i|\in\Z$ for all $i\in S$.
By Propositions \ref{P:lengthE} (2) and \ref{P:lowest} (2), 
we have 
\begin{align*}
&\rho(M^{0,\beta+kQ})-(1-\frac{1}{k})
\ge\dfrac{2}{3}\sum_{i\in S}|x_i|-\frac{1}{k}\sum_{i\in S}|x_i|^2-\frac{k-1}{k}\\
=&\frac{1}{k}\left(\sum_{i\in S}|x_i|(\frac{k}{2}-|x_i|)+\frac{k}{6}(\sum_{i\in S}|x_i|-6)+1\right). 
\end{align*}
If $\sum_{i\in S}|x_i|\ge6$, then  $\rho(M^{0,\beta+kQ})>1-({1}/{k})$, 
which is a contradiction.
So, we suppose that $\sum_{i\in S}|x_i|\le5$. Then $(|x_i|)_{i\in S}=(1,1,1)$, $(2,1,1)$, $(2,2,1)$, $(3,1,1)$, $(1,1,1,1)$ or $(2,1,1,1)$ up to permutation on coordinates.
By Proposition \ref{P:lengthE} (2), $\ell_{E_n}(\beta)\ge 2,3,4$ if $\sum_{i\in S}|x_i|=3,4,5$, respectively.
Let us consider the two cases (a) $k\ge3$ and (b) $k=2$.

(a) $k\ge3$. Then for each case, it is easy to see that
$$\rho(M^{0,\beta+kQ})\ge \ell_{E_n}(\beta)-\frac{1}{k}\sum_{i\in S}|x_i|^2>1-\frac{1}{k},$$
which is a contradiction.

(b) $k=2$. Then $(|x_i|)_{i\in S}=(1,1,1)$ or $(1,1,1,1)$; the latter case contradicts the condition $|\{i\in S\mid |x_i|=1\}|=1$ in (ii).
For the former case, if $\ell_{E_n}(\beta)>2$, then we obtain a contradiction using the same calculation as in (a).
Hence $\ell_{E_n}(\beta)=2$, and $\ell_{E_n}(\beta)=(2/3)\sum_{i\in S}|x_i|$.
By Proposition \ref{P:lengthE} (6), there exists $T\in H_m(4)$ such that $S\subset T$.
Let $\{q\}=T\setminus S$.
Set $\varepsilon_i=x_i/|x_i|$ if $i\in S\cap T=S$ and we choose $\varepsilon_q\in\{\pm1\}$ so that $(1/2)\sum_{i\in T}\varepsilon_i x_i\in Q$.
Then $\beta-\sum_{i\in T}\varepsilon_ix_i=- \varepsilon_qe_q$ is a root contained in $\beta+2Q$.

We now assume that $|S|\in\{1,2\}$ in (iii); $\beta=\sum_{i\in S}x_ie_i$, $S\subset\Omega_m$ with $0<|x_i|\le k/2$ for $i\in S$, and $|S|\in\{1,2\}$.
By Propositions \ref{P:lengthE} (3) and \ref{P:lowest} (2), 
\begin{align*}
&\rho(M^{0,\beta+kQ})-(1-\frac{1}{k})
\ge\sum_{i\in S}|x_i|-\frac{1}{k}\sum_{i\in S}|x_i|^2-\frac{k-1}{k}\\
=&\frac{1}{k}\left(\sum_{i\in S}(|x_i|-1)(k-1-|x_i|)+(|S|-1)(k-1)\right)\ge0.
\end{align*}
By the assumption $\rho(M^{0,\beta+kQ})=1-({1}/{k})$, all the equalities hold.
Hence $S=\{a\}$ and $|x_a|=1$.
Thus $\beta$ is a root.

Finally we assume (iv) for $E_6$; 
$\beta=\sum_{i\in S}x_ie_i$, $S=\{5,6,7\}$.
Note that $|x_i|\in\Z$ for $i\in\{5,6,7\}$.
We may assume that $S=\{a,b,c\}$ and $|x_a|\ge|x_b|\ge|x_c|\ge1$.
By $x_a+x_b+x_c=0$ and the condition (iv), we have $k>|x_a|=|x_b|+|x_c|$, and $k/2> |x_b|\ge|x_c|$.
By Propositions \ref{P:lengthE} (4) and \ref{P:lowest} (2), we obtain
\begin{align*}
&\rho(M^{0,\beta+kQ})-(1-\frac{1}{k})
\ge\sum_{i\in S}|x_i|-\frac{1}{k}\sum_{i\in S}|x_i|^2-\frac{k-1}{k}\\
=&\frac{1}{k}\left(2|x_b|(k-|x_b|-|x_c|)+2|x_c|(\dfrac{k}{2}-|x_c|)+k(|x_c|-1)+1\right)>0,
\end{align*}
which contradicts that $\rho(M^{0,\beta+kQ})=1-({1}/{k})$.
\end{proof}

\subsection{Type $F_4$}
Let $\Delta_{F_4}$ be the root system of type $F_4$ given as in \eqref{Eq:typeF} and let $Q=F_4$ be the root lattice of type $F_4$ generated by $\Delta_{F_4}$.
Then the sublattice $Q_L$ generated by long roots is isometric to a root lattice of type $D_4$.

\begin{lemma}\label{L:keyF} Let $k\in\Z_{\ge2}$ and set $Q=F_4$.
Let $\beta+kQ_L\in Q/kQ_L$ with $\beta\notin kQ_L$.
Then there exist $S\subset\Omega_4$ and $\sum_{i\in S}x_ie_i\in \beta+kQ_L$ satisfying one of the following :
\begin{enumerate}[{\rm (i)}]
\item $|S|\ge2$ and $0<|x_i|<k$ and $|x_i|\in\Z$ for all $i\in S$;
\item $S=\{a\}$ and $0<|x_a|\le k$ and $|x_a|\in\Z$;
\item $S=\Omega_4$ and $0< |x_i|< k$ and $|x_i|\in1/2+\Z$ for all $i\in S$.
\end{enumerate}
\end{lemma}
\begin{proof} If $\beta\in \bigoplus_{i=1}^4 \Z e_i$, then we can find a vector in $\beta+kQ_L$ satisfying (i) or (ii) by exactly the same argument as in Lemma \ref{L:keyB} for type $B_4$.
If $\beta\in Q\setminus \bigoplus_{i=1}^4 \Z e_i$, then $|x_i|\in1/2+\Z$ for all $ i\in\Omega_4$.
Since $k(e_i\pm e_j)\in kQ_L$ for $\{i,j\}\subset\Omega_4$, $\beta+kQ_L$ contains a vector satisfying (iii).
\end{proof}

\begin{proof}[Proof of Theorem \ref{Thm:key} for type $F_4$] We assume that $\rho(M^{0,\beta+kQ_L})=1-({1}/{tk})$ for $t\in\{1,2\}$.
Since $\rho(M^{0,kQ_L})=0$, we have $\beta\notin kQ_L$.
If $\beta+kQ_L$ contains a vector satisfying (i) or (ii) in Lemma \ref{L:keyF}, then we obtain the results by exactly the same argument for $B_4$.

The remaining case is (iii) in Lemma \ref{L:keyF}; we may assume that $\beta=\sum_{i\in\Omega_4}x_ie_i$ satisfies $0<|x_i|<k$ and $|x_i|\in 1/2+\Z$ for all $i\in\Omega_4$.
By Propositions \ref{P:length} (1) and \ref{P:lowest} (2), we obtain
\begin{align*}
&\rho(M^{0,\beta+kQ_L})-(1-\frac{1}{tk})\ge\ell_{F_4}(\beta)-\dfrac{|\beta|^2}{2k}-(1-\frac{1}{tk})\\
\ge&\frac{1}{2}\sum_{i\in\Omega_4}|x_i|-\frac{1}{2k}\sum_{i\in\Omega_4}|x_i|^2-\frac{tk-1}{tk}\\
=&\frac{1}{2k}\left(\sum_{i\in\Omega_4}(k-\frac{1}{2}-|x_i|)(|x_i|-\frac12)+\frac{2}{t}-1\right)\ge0.
\end{align*}
By the assumption $\rho(M^{0,\beta+kQ_L})=1-({1}/{tk})$, all the equalities hold.
Hence $t=2$, $|x_i|\in\{1/2,k-1/2\}$ for all $i\in\Omega_4$ and $\ell_{F_4}(\beta)=(1/2)\sum_{i\in\Omega_4}|x_i|$.
By Lemma \ref{L:lformula}, $|x_i|$ is constant for $i\in\Omega_4$.
Then $\beta+kQ_L$ contains a short root in $\{\sum_{i\in\Omega_4}c_ie_i\mid c_i\in\{\pm1/2\}\}$.
\end{proof}

\subsection{Type $G_2$}
Let $\Delta_{G_2}$ be the root system of type $G_2$ given as in \eqref{Eq:typeG} and let $Q=G_2$ be the root lattice generated by $\Delta_{G_2}$.
The sublattice $Q_L$ generated by long roots is isometric to a root lattice of type $A_2$.

\begin{lemma}\label{L:keyG} Let $k\in\Z_{\ge2}$ and set $Q=G_2$.
Let $\beta+kQ_L\in Q/kQ_L$ with $\beta\notin kQ_L$.
Then there exists $\sum_{i\in\Omega_3}x_ie_i\in \beta+kQ_L$ satisfying one of the following:
\begin{enumerate}[{\rm (i)}]
\item $|x_i|<k$ and $|x_i|\in\Z$ for all $i\in\Omega_3$;
\item $|x_i|\in(1/3)\Z\setminus\Z$ and $0<|x_i|< k$ for all $i\in\Omega_3$;
\end{enumerate}
\end{lemma}
\begin{proof} If $\beta\in Q_L$, then Lemma \ref{L:upperAD} (1) for $A_2$ shows that $\beta+kQ_L$ contains a vector satisfying (i).

Assume that $\beta=\sum_{i\in\Omega_3}x_ie_i\in Q\setminus Q_L$ for all $i\in\Omega_3$.
Then $|x_i|\in(1/3)\Z\setminus\Z$.
We may assume that $\Omega_3=\{a,b,c\}$ and $|x_a|\ge|x_b|\ge|x_c|>0$.
In addition, we assume $|x_a|\ge k$.
It follows from $x_a=-x_b-x_c$ that $x_ax_b<0$ and $x_ax_c<0$.
Then the sum of absolute values of entries of $\beta-(x_a/|x_a|)k(e_a-e_b)$ is less than that of $\beta$.
Repeating this procedure, we obtain a vector satisfying (ii).
\end{proof}

\begin{proof}[Proof of Theorem \ref{Thm:key} for type $G_2$] We assume that $\rho(M^{0,\beta+kQ_L})=1-({1}/{tk})$ for $t\in\{1,3\}$.
Since $\rho(M^{0,kQ_L})=0$, we have $\beta\notin kQ_L$.
If $\beta+kQ_L$ contains a vector satisfying (i) in Lemma \ref{L:keyG}, then $\langle\beta,\beta\rangle\in 2\Z$, and $t=1$ by Proposition \ref{P:lowest} (1).
Hence we obtain the results by exactly the same argument for $A_2$.

The remaining case is (ii) in Lemma \ref{L:keyG}; we may assume that $\beta=\sum_{i\in\Omega_3}x_ie_i$ satisfies $0< |x_i|<k$ and $|x_i|\in (1/3)\Z\setminus\Z$ for all $i\in\Omega_3$.
We may also assume that $\Omega_3=\{a,b,c\}$ and $|x_a|\ge|x_b|\ge|x_c|>0$.
Then $x_a+x_b+x_c=0$, and $|x_a|=|x_b|+|x_c|$.
Note that $k>|x_a|=|x_b|+|x_c|\ge2/3$.
Since $\sum_{i\in\Omega_3}|x_i|=2|x_a|\notin\Z$, we have $\ell_{G_2}(\beta)\ge(1/2)\sum_{i\in\Omega_3}|x_i|+(1/3)$ by Proposition \ref{P:length} (2).
By Proposition \ref{P:lowest} (2), we obtain
\begin{align*}
&\rho(M^{0,\beta+kQ_L})-(1-\frac{1}{tk})
\ge\frac{1}{2}\sum_{i\in\Omega_3}|x_i|+\frac{1}{3}-\frac{1}{2k}\sum_{i\in\Omega_3}|x_i|^2-\frac{tk-1}{tk}\\
=&\frac{1}{k}\left(\left(k-\frac{1}{3}-(|x_b|+|x_c|)\right)(|x_b|+|x_c|-\frac{2}{3})+(|x_b|-\frac13)(|x_c|-\frac13)+\frac{1}{t}-\frac13\right)\ge0.
\end{align*}
By the assumption $\rho(M^{0,\beta+kQ_L})=1-({1}/{tk})$, all the equalities hold.
Then $t=3$ and $|x_b|=|x_c|=1/3$.
Hence $|x_a|=2/3$, and $\beta$ is a short root.
\end{proof}

\begin{remark} Let $k\in\Z_{\ge2}$ and let $\g$ be a simple Lie algebra.
By Proposition \ref{simplecurrent} and the proof of Theorem \ref{Thm:key}, if $k$ is not $2$ or $\g$ is not of type $E_8$, then
$$\min\{\rho(M)\mid M\in\irr(K(\g,k))_{sc},\ M\not\cong K(\g,k)\}=1-\frac1k.$$ 
\end{remark}

\section{Automorphism groups of parafermion vertex operator algebras}\label{S:5}
Let $\g$ be a simple Lie algebra and let $k\in\Z_{\ge2}$.
We obtained in \eqref{Eq:phi}  the group homomorphism
\begin{equation*}
\varphi:\Aut(\Delta_\g)\to\Aut(K(\g,k)).\label{Eq:varphi}
\end{equation*}

When $k=2$ and $\g$ is simply laced, 
$K(\g,2)$ has been determined in \cite{LSY} (see also \cite{JLY}):
\begin{theorem}{\rm \cite{LSY}}\label{T:LSY} Let $\g$ be a simple Lie algebra of type $X_n\in\{A_n,D_n,E_n\}$.
Then $$\Aut(K(\g,2))\cong \begin{cases}\Aut(\Delta_\g)/\{\pm1\}& (X_n\in\{A_n,D_n,E_6,E_7\}),\\ Sp_8(2)& (X_n=E_8).\end{cases}$$
\end{theorem}
This theorem shows that if $\g$ is simply laced and $k=2$, then $\ker\varphi=\{\pm1\}$ and that $\Aut(K(\g,2))=\im\varphi$ if and only if the type of $\g$ is not $E_8$.

\begin{remark}
\begin{enumerate}
\item When the type of $\g$ is $A_1$ and $k\in\Z_{\ge3}$, $K(\g,k)$ is generated by the conformal vector and certain vector of conformal weight $3$ (\cite[Theorem 4.1]{DLWY}), which shows $\Aut(K(\g,k))\cong\Z/2\Z\cong\Aut(\Delta_\g)$.  
\item When the type of $\g$ is $A_2$ and $k\in\Z_{\ge3}$,  
it was shown in \cite[Theorem 3.6]{W} that $\Aut(K(\g,k))\cong(\Z/2\Z)\times\Sym_3\cong\Aut(\Delta_\g)$ by analysis on $K(\g,k)_2$ and $K(\g,k)_3$.  
\end{enumerate}
\end{remark}
\begin{remark}\label{R:level1} 
For a simple Lie algebra $\g$ of type $X_n$ and $t\in\Z_{>0}$, we use $K(X_n,t)$ to denote $K(\g,t)$.
It is known that $K(X_n,1)=\C$ if $X_n$ is simply laced, $K(B_n,1)\cong K(A_1,2)$ $(n\ge2)$, $K(C_n,1)\cong K(A_{n-1},2)$ ($n\ge2$), $K(F_4,1)\cong K(A_2,2)$ and $K(G_2,1)\cong K(A_1,3)$ (see \cite[Section 5]{DW2}).
Hence $\Aut(K(X_n,1))=1$ if $X_n$ is simply laced or $X_n=B_n$, $\Aut(K(C_n,1))\cong \Sym_{n}$ $(n\ge2)$, $\Aut(K(F_4,1))\cong \Sym_3$ and $\Aut(K(G_2,1))\cong \Z/2\Z$.
In particular, $\Aut(K(\g,1))=\im\varphi$ for any $\g$.
\end{remark}
\begin{lemma}\label{L:T1} Let $\g$ be a simple Lie algebra and let $k$ be a positive integer.
Assume that one of the following holds: (1) $k\ge3$; (2) $k=2$ and $\g$ is non simply laced.
Then the group homomorphism $\varphi$ in \eqref{Eq:phi} is injective.
\end{lemma}
\begin{proof} By Lemma \ref{L:inj0}, it suffices to check that $\Aut(\Delta_\g)$ acts faithfully on $Q/kQ_L$.
If (1) holds (resp. (2) holds), then, by Lemma \ref{L:minnorm}, for a root $\alpha$ (resp. for a short root $\alpha$), $(\alpha+kQ_L)\cap\Delta_\g=\{\alpha\}$.
Let $g\in\Aut(\Delta_\g)$ which acts trivially on $Q/kQ_L$.
Then $g$ fixes all roots in $\Delta_\g$ (resp. all short roots in $\Delta_\g$), and $g=id$ since $\R\Delta_\g$ is spanned by roots (resp. short roots) if (1) holds (resp. (2) holds).
\end{proof}
The following is the main theorem of this article.

\begin{theorem}\label{Thm:main2} Let $\g$ be a simple Lie algebra and let $k$ be a positive integer.
Assume that one of the following holds: (1) $k\ge3$; (2) $k=2$ and the type of $\g$ is non simply laced.
Then $\Aut(K(\g,k))\cong \Aut(\Delta_\g)$.
\end{theorem}
\begin{proof}
By Lemma \ref{L:T1} and $|\Aut(\Delta_\g)|<\infty$, it suffices to construct an injective group homomorphism from $\Aut(K(\g,k))$ to $\Aut(\Delta_\g)$.

Set $\Delta=\Delta_\g$ if (1) holds.
Let $\Delta$ be the set of all short roots if (2) holds; in this case, $\Delta$ is a (scaled) simply laced root system (see Lemma \ref{L:auts} (2)).
Recall from Proposition \ref{simplecurrent} that 
$$\Irr(K(\g,k))_{sc}=\{M^{0,\beta+kQ_L}\mid \beta+kQ_L\in Q/kQ_L\}.$$
By Corollary \ref{C}, $\Aut(K(\g,k))$ preserves $\{M^{0,\alpha+kQ_L}\mid \alpha\in\Delta\}$.

Let $\sigma\in \Aut(K(\g,k))$.
By Lemma \ref{L:minnorm}, for any $\alpha\in\Delta$, $(\alpha+kQ_L)\cap\Delta=\{\alpha\}$;
we define $\tilde\sigma(\alpha)\in\Delta$ so that $M^{0,\tilde\sigma(\alpha)+kQ_L}\cong M^{0,\alpha+kQ_L}\circ \sigma^{-1}$.
Then $\tilde\sigma$ defines a permutation on $\Delta$, i.e., $\tilde{\sigma}\in \Sym\Delta$, the permutation group on $\Delta$.
By Proposition \ref{Cor:fusion}, for $\alpha\in\Delta$, $$M^{0,\alpha+kQ_L}\boxtimes M^{0,-\alpha+kQ_L}=M^{0,kQ_L}.$$
Applying the $\sigma^{-1}$-conjugation to both sides, we obtain $M^{0,\tilde\sigma(\alpha)+kQ_L}\boxtimes M^{0,\tilde\sigma(-\alpha)+kQ_L}=M^{0,kQ_L}$ by \eqref{Eq:rhoconf}.
On the other hand, $M^{0,\tilde\sigma(\alpha)+kQ_L}\boxtimes M^{0,-\tilde\sigma(\alpha)+kQ_L}=M^{0,kQ_L}.$
Since $\Irr(K(\g,k))_{sc}$ has a group structure under the fusion product, we obtain $M^{0,\tilde\sigma(-\alpha)+kQ_L}\cong M^{0,-\tilde\sigma(\alpha)+kQ_L}$, and 
\begin{equation}
\tilde\sigma(-\alpha)=-\tilde\sigma(\alpha).\label{Eq:sigma-}
\end{equation}
By \eqref{Eq:rhoconf} and Proposition \ref{Cor:fusion}, for $\alpha_1,\alpha_2\in\Delta$
$$\rho(M^{0,\alpha_1+\alpha_2+kQ_L})=\rho((M^{0,\alpha_1+kQ_L}\boxtimes M^{\alpha_2+kQ_L})\circ\sigma^{-1})=\rho(M^{0,\tilde\sigma(\alpha_1)+\tilde\sigma(\alpha_2)+kQ_L}).$$
By Proposition \ref{P:lowest} (1), we obtain for $i=1,2$,
$$\dfrac{|\alpha_i|^2}{2k}\equiv\dfrac{|\tilde\sigma(\alpha_i)|^2}{2k},\quad \dfrac{|\alpha_1+\alpha_2|^2}{2k}\equiv\dfrac{|\tilde\sigma(\alpha_1)+\tilde\sigma(\alpha_2)|^2}{2k}\pmod\Z.$$
Then we obtain 
\begin{equation}
\dfrac{\langle \alpha_1,\alpha_2\rangle}{k}\equiv\dfrac{\langle \tilde\sigma(\alpha_1),\tilde\sigma(\alpha_2)\rangle}{k}\pmod\Z.\label{Eq:mod}
\end{equation}
By \eqref{Eq:sigma-} and \eqref{Eq:mod}, we can apply Lemma \ref{L:innerZ} to $\tilde{\sigma}$; for any $\alpha_1,\alpha_2\in\Delta$,
$$\langle \alpha_1,\alpha_2\rangle=\langle\tilde\sigma(\alpha_1),\tilde\sigma(\alpha_2)\rangle.$$
By Lemma \ref{L:symdelta}, 
$\tilde\sigma$ can be extended to an element in $\Aut(\Delta)$.
Thus we obtain the group homomorphism
$$\nu:\Aut(K(\g,k))\to\Aut(\Delta),\quad \sigma\mapsto\tilde\sigma.$$
Since $\Delta$ generates $Q$, $\ker\nu$ acts trivially on $\irr(K(\g,k))_{sc}$ by Proposition \ref{simplecurrent}.
By Lemma \ref{L:inj}, $\ker\nu=\{id\}$, and $\nu$ is injective.

If (1) holds, then $\Delta=\Delta_\g$, and $\Aut(\Delta)=\Aut(\Delta_\g)$.
If (2) holds and $\g$ is not of type $C_4$, then by Lemma \ref{L:auts}, $\Aut(\Delta)=\Aut(\Delta_\g)$.
For $k=2$ and $\g$ is of type $C_4$, 
$\im\nu$ preserves $\{\beta+kQ_L\in Q/kQ_L\mid (\beta+kQ_L)\cap\Delta_\g\neq\emptyset\}$ by Corollary \ref{C} 
and it also preserves  $\Delta_\g$.
Hence $\im\nu\subset\Aut(\Delta_\g)$.

Thus, in all cases, we obtain an injective group homomorphism from $\Aut(K(\g,k))$ to $\Aut(\Delta_\g)$.
\end{proof}

Combining Theorems \ref{T:LSY} and \ref{Thm:main2} and Remark \ref{R:level1}, we obtain the following corollary:

\begin{corollary} Let $\g$ be a simple Lie algebra and $k\in\Z_{>0}$.
Then $\Aut(K(\g,k))$ is larger than $\im\varphi$ if and only if $\g$ is of type $E_8$ and $k=2$.
\end{corollary}

\vskip.25cm
\paragraph{\bf Acknowledgement.} The authors thank Professor Hiroshi Yamauchi  and the referee for useful comments.

\end{document}